\documentclass{amsart}
\usepackage{latexsym,amssymb,enumerate, amsmath}

\renewcommand{\o}{\overline}

\newenvironment{enumeratei}{\begin{enumerate}[\upshape (a)]}
    {\end{enumerate}}
\newenvironment{enumeraten}{\begin{enumerate}[\upshape (i)]}
    {\end{enumerate}}

\newtheorem{theorem}{Theorem}[section]

\newtheorem{lemma}[theorem]{Lemma}
\newtheorem{proposition}[theorem]{Proposition}
\newtheorem*{thmA}{Theorem A}

\newtheorem*{thmC}{Theorem C}
\newtheorem*{corB}{Corollary B}

\theoremstyle{definition}

\newtheorem{rem}[theorem]{Remark}

\newcommand{\C}{{\mathbb C}}
\newcommand{\K}{{\mathbb K}}
\newcommand{\F}{{\mathbb F}}
\def\irr#1{{\rm Irr}(#1)}

\def\cd#1{{\rm cd}(#1)}

\def\cent#1#2{{\bf C}_{#1}(#2)}

\def\syl#1#2{{\rm Syl}_#1(#2)}
\def\nor{\trianglelefteq\,}
\def\norm#1#2{{\bf N}_{#1}(#2)}

\def\oh#1#2{{\bf O}_{#1}(#2)}
\def\zent#1{{\bf Z}(#1)}
\def\hyperzent#1{{\bf {Z_{\infty}}}(#1)}
\def\sbs{\subseteq}

\def\fit#1{{\bf F}(#1)}
\def\fitt#1{{\bf F}_2(#1)}
\def\diam#1{{\rm diam}(#1)}
\def\frat#1{{\bf \Phi}(#1)}
\def\GF#1{{\rm GF}(#1)}

\def\V#1{{\rm V}(#1)}
\def\E#1{{\rm E}(#1)}

\def\nfr#1{\widehat{#1}^{*}}

\begin{document}

\title{Groups whose character degree graph has diameter three}

\author[C. Casolo et al.]{Carlo Casolo}
\address{Carlo Casolo, Dipartimento di Matematica e Informatica U. Dini,\newline
Universit\`a degli Studi di Firenze, viale Morgagni 67/a,
50134 Firenze, Italy.}
\email{carlo.casolo@unifi.it}

\author[]{Silvio Dolfi}
\address{Silvio Dolfi, Dipartimento di Matematica e Informatica U. Dini,\newline
Universit\`a degli Studi di Firenze, viale Morgagni 67/a,
50134 Firenze, Italy.}
\email{dolfi@math.unifi.it}

\author[]{Emanuele Pacifici}
\address{Emanuele Pacifici, Dipartimento di Matematica F. Enriques,
\newline Universit\`a degli Studi di Milano, via Saldini 50,
20133 Milano, Italy.}
\email{emanuele.pacifici@unimi.it}

\author[]{Lucia Sanus}
\address{Lucia Sanus, Departament d'\`Algebra, Facultat de
  Matem\`atiques, \newline
Universitat de Val\`encia,
46100 Burjassot, Val\`encia, Spain.}
\email{lucia.sanus@uv.es}

\dedicatory{Dedicated to the memory of Laci Kov\'acs}
\thanks{The first three authors are partially supported by the Italian INdAM-GNSAGA. The fourth author is partially supported
by  the Spanish  MINECO proyecto MTM2013-40464-P,  partly with  FEDER funds and Prometeo2011/030-Generalitat Valenciana.}


\begin{abstract}
Let \(G\) be a finite group, and let \(\Delta(G)\) denote the \emph{prime graph} built on the set of  degrees of the irreducible complex characters of \(G\). It is well known that, whenever \(\Delta(G)\) is connected, the diameter of \(\Delta(G)\) is at most \(3\). In the present paper, we provide a description of the finite solvable groups for which the diameter of this graph attains the upper bound. This also enables us to confirm a couple of conjectures proposed by M.L. Lewis.
\end{abstract}

\maketitle

\section{Introduction}

Let  $G$ be a finite group; we denote by $\irr G$ the set of all irreducible complex characters of $G$, and write
$${\rm cd}(G) = \{ \chi (1)\mid \chi \in  \irr G\}$$
for the set of the degrees of such characters. The character degree graph $\Delta (G)$ is thus defined as the graph with vertex set the set $\rho (G)$ of all the primes that divide some $\chi (1)\in {\rm cd}(G)$, and two distinct primes $p$ and $q$ are adjacent if and only if $pq$ divides some degree in cd$(G)$.
The study of the graph $\Delta (G)$ and of the relationships between the properties of $\Delta (G)$ and the structural features of the group $G$, has by now a rich literature (we recommend the survey paper \cite{Lew0} for a general  overview of the subject), and the purpose of this paper is to contribute to one particular aspect of this research.

A fundamental result of P.P. Palfy (\cite{PPP1}) ensures that if $G$ is a solvable group, then given any three distinct primes in $\rho (G)$, at least two of them are adjacent in $\Delta (G)$.  From this it immediately follows that, for a solvable group $G$, $\Delta (G)$ has at most two connected components (both inducing a complete subgraph of $\Delta (G)$), and that, when $\Delta (G)$ is connected, the diameter of  $\Delta (G) $ is at most  $3$ (that this latter inequality holds in any finite group is proved in \cite{LW}).
For some time it has been unknown whether there existed solvable groups whose character degree graph has diameter $3$, until  the question was settled by  M.L. Lewis, who constructed in \cite{Lew2} a solvable group $G$ such that $\Delta (G)$ has $6$ vertices and diameter $3$.

It was Lewis construction and his related comments (made particularly explicit in \cite{Lew0}) that prompted us to study in more detail solvable groups $G$  such that  the diameter of $\Delta (G)$ is   $3$.
Through our analysis in the present paper, we pin their structure down enough to show that they all closely resemble Lewis' examples, and to allow to confirm a couple of conjectures appearing in \cite{Lew0}.

In the following statement, which is the main result of this paper, $\fit G$ denotes the Fitting subgroup of the group $G$ and, for $i\ge 1$, $\gamma_i(P)$ is the $i$-th term of the lower central series of  the \(p\)-group $P$.

\begin{thmA}  Let $G$ be a finite solvable group such that $\Delta (G)$ is connected and  $\diam{\Delta (G)} = 3$. Then the following conclusions hold.
\begin{enumeratei}
\item There exists a prime $p$ such that $G = PH$, with $P$ a normal non-abelian Sylow $p$-subgroup of $G$ and $H$ a $p$-complement.
\item $\fit G= P\times A$, where $A = \cent HP \le \zent G$, $H/A$ is not nilpotent and has cyclic Sylow subgroups.
\item  $\Delta (G/\gamma_3(P))$ is disconnected.
\item  If $c$ is the nilpotency class of $P$, then all factors $M_1 = [P,G]/P'$ and $M_i = \gamma_i(P)/\gamma_{i+1}(P)$, for $2\le i\le c$, are chief factors of $G$ of the same order $p^n$, where $n$ is divisible by at least two distinct odd primes; moreover, for all $1\le i\le c$,  $G/\cent G{M_{i}}$ embeds as an irreducible subgroup in the group of semi-linear transformations $\Gamma (p^n)$.
\end{enumeratei}
\end{thmA}

In particular,  we have that  $\Delta (G/\gamma_3(P))$ is a disconnected subgraph of $\Delta (G)$
with the same set of vertices, thus confirming a suggestion of Lewis (\cite{Lew0}).  In fact, it will not be hard to derive a proof of a related conjecture concerning the structure of the graph $\Delta (G)$, when $G$ is solvable and $\diam{\Delta(G)} = 3$. Let $r,s$ be two vertices of $\Delta (G)$ with $d(r,s) = 3$, and denote by $\pi_1$ and $\pi_2$ the sets consisting of $r$, respectively $s$, and all vertices adjacent to it; then $\pi_1\cap \pi_2 = \emptyset$ and (by Palfy's three primes condition)  $\rho (G) = \pi_1 \cup \pi_2$. Moreover, denoting by \(\fitt G\) the second Fitting subgroup of \(G\) and supposing $p\in \pi_1$ (we will see that $r\ne p\ne s$), we will show that $\pi_1 = \pi (\fitt G/\zent G)$, $\pi_2 = \pi (G/\fitt G)$,  $2\not\in \pi_2$, $|\pi_2|\ge 2$ and $\pi_1$, $\pi_2$ both induce complete subgraphs of $\Delta (G)$ (see Remark~\ref{spiegazione}). Indeed, $\pi_1$ and $\pi_2$ are the set of vertices of the two connected components of $\Delta (G/\gamma_3(P))$; it follows (see Remark~\ref{spiegazione}) that
$$|\pi_1| \ge 2^{|\pi_2|} .$$
Hence, Conjecture 4.8 of \cite{Lew0}  is established; in particular, $|\rho (G)| \ge 6$ (indeed,  it turns out that Lewis' example has the smallest possible order).

Still in the spirit of another suggestion by Lewis (see the paragraph following 5.8 in \cite{Lew0}), an immediate consequence of Theorem A is the following result.

\begin{corB}
Let \(G\) be a finite solvable group, and assume that \(\Delta(G)\) is connected with diameter \(3\). Then the Fitting height of \(G\) is precisely \(3\). In fact, \(G\) is a nilpotent-by-metacyclic group.
\end{corB}

As it is apparent from the above remarks, a central role in our treatment is played by solvable groups with disconnected degree graph, the main features of those we will need to have almost constantly in hand. For this, our main source is their description in \cite{L} (although similar results also appear in~\cite{Zh} and~\cite{PPP2}), and we in particular refer to the list of six subcases in section 2 (and 3) of that paper. From the same arguments that prove Theorem A, we derive a result which we believe adds to the understanding of case 2.6 in \cite{L}.

\begin{thmC}  Let $G$ be a finite solvable group such that $\Delta (G)$ is disconnected and $\fit G$ is not abelian. Then there is a unique prime $p$ such that $P = \oh pG$ is not contained in $\zent G$ and
\begin{enumeratei}
\item either $p$ is an isolated vertex of $\Delta (G)$, or
\item  $\Delta (G/P')$ is disconnected and, if $c$ is the nilpotency class of $P$, all factors $M_1 = [P,G]/P'$ and $M_i = \gamma_i(P)/\gamma_{i+1}(P)$, for $2\le i\le c$, are chief factors of $G$ of the same order $p^n$, with $n \geq 3$; moreover, for all $1\le i\le c$,  $G/\cent G{M_{i}}$ embeds as an irreducible subgroup in the group of semi-linear transformations $\Gamma (p^n)$.
\end{enumeratei}
\end{thmC}
We conjecture that, both in the disconnected and in the diameter-three case, the chief factors $M_i$ are pairwise non-isomorphic as \(G\)-modules over ${\rm{GF}}(p)$; this is true for the first pair $M_1$ and $M_2$ (see point (a) in the proof of Proposition \ref{maxiemendamento}), but we were not able to prove it in general. Another question that we leave open is whether, in both Theorem~A and Theorem~C, one has $P = [P,G]\times \zent G$ (again, this is true modulo $\gamma_3(P)$). Finally, by looking at the known examples, one might ask if it is true that, in Theorem A, not only $\Delta(G/\gamma_{3}(P))$ but also $\Delta(G/\gamma_{c}(P))$ is disconnected.


\section{Notation and preliminaries}

Throughout this paper, every group is tacitly assumed to be a finite group. We write $\V G$ and $\E G$ for the sets of vertices and edges,
respectively, of the prime graph $\Delta(G)$ on irreducible character degrees. We denote by $d_G(u,v)$ the distance  in $\Delta(G)$
between the two (distinct) vertices $u$ and \(v\) (i.e. the length of a shortest path joining $u$
and $v$; set $d_G(u,v) = \infty$ if there is no such path), and by  $\diam {\Delta (G)}$ the maximum of $d_G(u,v)$  for   $u,v \in \V G$ if ${\Delta (G)}$ is connected (whereas we set $\diam {\Delta (G)}=\infty$  if $\Delta(G)$ is not connected).

 As customary, we denote  by $\Gamma(p^n)$ the semi-linear group on  the field $\GF{p^n}$, and by $\Gamma_0(p^n)$ the subgroup of $\Gamma (p^n)$ induced by the field multiplications. If \(V\) is an \(n\)-dimensional vector space over \({\rm{GF}}(p)\), then \(V\) can be identified with the additive group of a field of order \(p^n\), and in this sense we write \(\Gamma(V)\) and \(\Gamma_0(V)\) for \(\Gamma(p^n)\) and \(\Gamma_0(p^n)\) respectively.

Let $a>1$ and $n$ be positive integers.  A prime $t$ is called a \emph{primitive prime divisor} for  $(a,n)$ if $t$ divides $a^n-1$ but $t$ does not  divide $a^j-1$ for $1\leq j< n$. Recall that,  by a well-known result by Zsigmondy (\cite[Theorem 6.2]{MW}), such a prime always exists except when $n=6$ and $a=2$, or $n=2$ and $a+1$ is a power of $2$.

Let $N$ be a normal subgroup of $G$ and let $\lambda \in \irr N$. We denote by $\irr{G|\lambda}$ the set of irreducible characters $\chi$ of $G$ such that $\lambda$ is an  irreducible constituent of $\chi_N$. In this setting, \(\chi\) and \(\lambda\) are said to be \emph{fully ramified} with respect to \(G/N\) (but sometimes, when the context is clear enough, we also say that \(\lambda\) is fully ramified in \(G\)) if \(\chi_N=e\lambda\) with \(e^2=|G:N|\). By \cite[Problem 6.3]{I}, this is equivalent to the fact that \(\chi\) vanishes on \(G\setminus N\) with \(\lambda\) invariant in \(G\), and also to the fact that \(\chi\) is the unique irreducible constituent of \(\lambda^G\) still with \(\lambda\) invariant in \(G\).

If \(A\) is an abelian group, we write \(\widehat A\) to denote the dual group of A, that is, the set \(\irr A\) endowed with multiplication of characters.

Also, we freely use without references some basic  facts of Character Theory  such as  Clifford Correspondence, Gallagher's Theorem, Ito-Michler's Theorem, results concerning character extension and coprime actions (see \cite{I}).

 We shall also take into account the following well-known result concerning character degrees.

 \begin{lemma}[\mbox{\cite[Proposition 17.3]{MW}}]
\label{brodkey} Let $G$ be a solvable group. Let $F = \fit G$ and $K = {\bf F}_2(G)$. Then there exists $\chi\in\irr G$
such that $\pi(K/F) \sbs \pi(\chi(1))$.
\end{lemma}

As mentioned in the Introduction, we will make an intensive use of the classification, provided in~\cite{L}, of solvable groups whose character degree graph is disconnected. The next statement summarizes some aspects of that classification: the groups in (a), (b) and (c) are respectively those of types 2.1, 2.4 and 2.6 (described further in 3.1, 3.4 and 3.6, respectively) in \cite{L}.

\begin{theorem}\label{class} Let $G$ be a solvable group, and set \(F=\fit G\), \(K=\fitt G\). Assume that $\Delta (G)$ has two connected components. Then the following conclusions hold.
\begin{enumeratei}
\item Assume that $G$ is metanilpotent. Then \(G=PH\), where \(P\nor G\) is a non-abelian Sylow $p$-subgroup for a suitable prime \(p\), and \(H\) is an abelian $p$-complement. Moreover, $P'\leq\cent P H$, and every non-linear irreducible character of $P$ is fully ramified with respect to $P/\cent P H$. Finally, the sets of vertices of the two connected components of  \(\Delta(G)\) are respectively \(\{p\}\) and \(\pi(G/F)\). 
\item Assume that \(F\) is abelian, and that \(|\V G|>2\). Then \(G=MH\), where \(M\nor G\) is an elementary abelian \(p\)-group for a suitable prime \(p\), and \(H\) is a complement for \(M\). Also, \(F=M\times \zent G\), \(\zent G=\cent H M\) and $G/F \leq \Gamma(M)$. The subgroup \(K\) acts irreducibly (by conjugation) on \(M\), and both \(K/F\) and \(G/K\) are cyclic groups. Finally, the sets of vertices of the two connected components of  \(\Delta(G)\) are respectively \(\pi(K/F)\) and \(\pi(G/K)\).   
\item Assume that \(F\) is non-abelian and that, whenever \(\oh r G\) is non-abelian, the prime \(r\) is not an isolated vertex of \(\Delta(G)\). Then \(G=PH\), where \(P\nor G\) is a non-abelian Sylow $p$-subgroup for a suitable prime \(p\), and \(H\) is a $p$-complement. Also,  \(F=P\times U\) where \(U\leq\zent G\). The factor group \(G/P'\) is a group as in \((b)\), so, in particular, \(K/F\) and \(G/K\) are cyclic groups. Finally, the sets of vertices of the two connected components of  \(\Delta(G)\) are respectively \(\{p\}\cup\pi(K/F)\) and \(\pi(G/K)\). 
\end{enumeratei}
\end{theorem}

We stress that a group \(G\) as in (b) or (c) of Theorem~\ref{class}  is such that every Sylow subgroup of \(G/F\) is cyclic.

We also quote the following result, which is Theorem~5.5 of \cite{L}.

\begin{lemma}
\label{lewis} Let $G$ be a solvable group such that $\Delta(G)$ is a
disconnected graph. Then there exists a unique prime $p$ such
that $\oh p G$ is non-central in $G$.
\end{lemma}

Next, another preliminary lemma.

\begin{lemma}
  \label{U}
Let $G$ be a group such that $\fit G = M \times Z$, with $Z \leq \zent G$  and $M \nor G$. 
 Assume also that every irreducible
character of $\fit G$ extends to its inertia subgroup. Then
$\Delta(G) = \Delta(G/Z)$.
\end{lemma}

\begin{proof} Observe first that, by our assumptions,  every irreducible character of \(Z\) has an extension to \(G\): in fact, if \(\theta\) is in \(\widehat{Z}\), then \(\theta\times 1_M\in\irr{\fit G}\) extends to \( G=I_G(\theta\times 1_M)\). Now, let \(d\) be a number in \(\cd G\), \(\chi\) an irreducible character of \(G\) of degree \(d\), and  \(\theta\) an irreducible constituent of \(\chi_Z\);  denoting by \(\xi\) an extension of \(\theta\) to \(G\), by Gallagher's Theorem there exists \(\psi\in\irr{G/Z}\) such that \(\chi=\xi\psi\). As a consequence, \(d=\chi(1)=\psi(1)\in\cd{G/Z}\), and the desired conclusion follows.
\end{proof}

Finally, the following result  by C.P. Morresi Zuccari (\cite[Corollary C]{Z})  will also be relevant for our purposes.

\begin{theorem} Let $G$ be a solvable group such that $\Delta (G)$ is connected.  If $\fit G$ is abelian, then $\diam{\Delta(G)}\leq 2$.
  \label{Zuccari}

\end{theorem}

\section{Some proofs}

We start with two lemmas concerning modules over finite fields for cyclic groups.

\begin{lemma} Let \(G\) be a cyclic group, \(\K\) a finite field of order $q$, and \(M\) a faithful irreducible $m$-dimensional \(\K [G]\)-module. Also, let \(\epsilon\) be an element of order $|G|$ in the multiplicative group of $\F={\rm{GF}}(q^m)$. Then the following hold.

\begin{enumeratei}
\item If \(M\) is a constituent of \(M\otimes_{\K} M\), then there exist \(\sigma_1\) and \(\sigma_2\) in \({\rm{Gal}}(\F\,|\,\K)\) such that \(\epsilon^{\sigma_1}\cdot\epsilon^{\sigma_2}=\epsilon\).
\item If \(M\) is self-contragredient, then there exists \(\sigma\) in \({\rm{Gal}}(\F\,|\,\K)\) such that \(\epsilon^{\sigma}=\epsilon^{-1}\).
\end{enumeratei}
\label{tensor}
\end{lemma}

\begin{proof} Observe that \(\F\) is a splitting field for \(G\) over \(\K\). By~\cite[II.3.10]{H}, we can identify $M$ with the additive group of $\F$, and the action of a
 suitable generator \(x\) of $G$ with the multiplication by $\epsilon$. We denote by  \(M_{\F}\)
 the $1$-dimensional \(\F [G]\)-module arising in this way. Setting \(M^{\F}=M\otimes_{\K}\F\), by \cite[VII, 1.16 a)]{HB} we get \[M^{\F}=\bigoplus_{\sigma\in{\rm{Gal}}(\F\,|\,\K)}(M_{\F})^{\sigma}.\]

Now, if \(M\) is a constituent of \(M\otimes_{\K} M\), then
\(M^{\F}\) is a direct summand (as an \(\F [G]\)-module) of \[(M\otimes_{\K} M)^{\F}\simeq
M^{\F}\otimes_{\F} M^{\F}\simeq
\bigoplus_{\sigma_1,\sigma_2\in{\rm{Gal}}(\F\,|\,\K)}(M_{\F})^{\sigma_1}\otimes_{\F}(M_{\F})^{\sigma_2}.\]

\noindent In particular, there exist \(\sigma_1\) and \(\sigma_2\)
in \({\rm{Gal}}(\F\,|\,\K)\) such that
\(M_{\F}\simeq(M_{\F})^{\sigma_1}\otimes_{\F}(M_{\F})^{\sigma_2}\).
Considering now the action of \(x\) on these two isomorphic \(\F
[G]\)-modules, claim (a) follows.

As for (b), if \(M\) is \(\K [G]\)-isomorphic to its contragredient
module \(M^*\), then we get
\[\bigoplus_{\sigma\in{\rm{Gal}}(\F\,|\,\K)}(M_{\F})^{\sigma}\simeq
M^{\F}\simeq
(M^*)^{\F}\simeq\bigoplus_{\sigma\in{\rm{Gal}}(\F\,|\,\K)}((M^*)_{\F})^{\sigma}.\]
In particular,  there exists \(\sigma \) in
\({\rm{Gal}}(\F\,|\,\K)\) such that \((M_{\F})^{\sigma}\) is \(\F
[G]\)-isomorphic to \((M^*)_{\F}\), which is in turn \(\F
[G]\)-isomorphic to \((M_{\F})^*\). Claim (b) is now achieved by
comparing the action of \(x\) on the two relevant \(\F [G]\)-modules.
\end{proof}

\begin{lemma} Let \(G\) be a cyclic group, \(p\) a prime, and \(M\) a
 faithful irreducible \({\rm{GF}}(p)[G]\)-module. Setting \(|M|=p^m\), assume that there exists a divisor \(r\) of
\(m\), $1 \leq r < m$,  such that \(\frac{p^m-1}{p^r-1}\) divides \(|G|\). Then the following hold.

\begin{enumeratei}
\item If \(M\) is a constituent of \(M\wedge_{{\rm{GF}}(p)} M\), then \((p,m/r)=(2,2)\).
\item If \(M\) is self-contragredient, then \((|G|,m/r)=(p^r+1,2)\).
\end{enumeratei}
\label{modules}
\end{lemma}

\begin{proof}
Let us first prove Claim (a). If \(M\) is a constituent of
\(M\wedge_{{\rm{GF}}(p)} M\), then it is clearly a constituent of
\(M\otimes_{{\rm{GF}}(p)} M\) as well. Therefore, by
Lemma~\ref{tensor}(a), there exist \(a, b\in\{0,...,m-1\}\) (say
\(a\geq b\)) such that \(p^a+p^b\equiv 1\;({\rm mod}\;|G|)\); in
particular, setting \(t=\frac{p^m-1}{p^r-1}\), we have that \(t\)
divides \(p^a+p^b-1\).

Since we have \(t>p^{m-r}\), we also have \(a\geq m-r\); in fact, assuming the contrary, we would get \(p^b\leq p^a\leq p^{m-r-1}\), thus \(p^a+p^b-1\leq 2p^{m-r-1}-1<p^{m-r}<t\), contradicting the fact that \(t\) divides \(p^a+p^b-1\). Therefore
we can write \(a=m-n\) where \(n\) lies in \(\{1,..., r\}\). Now,
\(t\) is a divisor of \((p^{m-n}+p^b-1)\cdot
p^n-(p^m-1)=p^{b+n}-p^n+1\), whence \(m-r\leq b+n\leq a+n=m\).
Again, defining \(\ell=m-(b+n)\), we get \(0\leq\ell\leq r\) and
\(t\) divides \((p^{b+n}-p^n+1)\cdot
p^{\ell}-(p^m-1)=p^{\ell}-p^{n+\ell}+1\). As a result, \(t\) is a
divisor of \(p^{\ell}\cdot(p^n-1)-1\).

On the other hand, if \(m/r\geq 3\), then \(m-r\geq 2r\), and so
\(n+\ell\leq 2r\leq m-r\). This implies \(p^{n+\ell}-p^{\ell}<
p^{m-r}< t\), and now the only possibility is
\[p^{\ell}\cdot(p^n-1)=1,\] which yields \(\ell=0\), \(p=2\),
\(n=1\). So, \(p^a=p^b=2^{m-1}\). Setting \(\F=\GF{2^m}\) and denoting by \(\sigma\) the element
of \({\rm{Gal}}(\F\,|\,{\rm{GF}}(2))\) which maps every \(f\in\F\)
to \(f^{2^{m-1}}\), the conclusion so far is that \((M_{\F})^{\sigma}\otimes_{\F} (M_{\F})^{\sigma}\) is the unique constituent of \(M^{\F}\otimes_{\F}M^{\F}\) which is isomorphic to \(M_{\F}\). But \((M_{\F})^{\sigma}\otimes_{\F} (M_{\F})^{\sigma}\) is not a
constituent of \(M^{\F}\wedge_{\F} M^{\F}\simeq (M\wedge_{\GF{2}} M)^{\F}\), against the fact that \(M\) is a constituent of \(M\wedge_{\GF{2}} M\).

It remains to treat the case \(m/r=2\), whence \(t=p^r+1\) divides
\(p^a+p^b-1\) with \(0< a\leq 2r-1\). Note that we must have \(a\geq
r\), so we can write \(a=r+n\) with \(0\leq n<r\). Now, \(t\)
divides \(p^{n+r}+p^b-1-p^n\cdot(p^r+1)=-p^n+p^b-1\). In particular,
\(t\leq |-p^n+p^b-1|<p^n+p^b-1\). This in turn implies \(b\geq r\),
therefore we write \(b=r+k\) where \(0\leq k<r\). Finally, \(p^r+1\)
divides \(-p^n+p^{r+k}-1-p^k\cdot(p^r+1)=-p^n-p^k-1\), whence
\(p^r+1\leq p^n+p^k+1\leq 2p^{r-1}+1\). It follows that \(p=2\), and
(a) is proved.

\smallskip
We move now to Claim (b). If \(M\) is self-contragredient, then
Lemma~\ref{tensor}(b) yields that there exists \(k\in\{0,...,m-1\}\)
such that \({p^k}\equiv -1\;({\rm{mod}}\;|G|)\). Therefore,
\(\frac{p^m-1}{p^r-1}\) divides \(p^k+1\). Let us first exclude the
possibility \(k=0\); in that case, in fact, we would have \(|G|=2\)
and \(m=1\), contradicting the existence of a proper positive
divisor of \(m\). Observe also that, as by~\cite[II.3.10]{H} $m$ is the smallest
positive integer such that $p^m \equiv 1 \pmod{|G|}$,  for
every integer \(z\neq 0\) such that \(|G|\) divides \(p^z-1\) we get
\(m\mid z\), so \(m\) divides \(2k\); on the other hand, since \(0<
k\leq m-1\), we have in fact \(m=2k\).

Our conclusion so far is that \(\frac{p^{2k}-1}{p^r-1}\) is a
divisor of \(p^k+1\); this yields \(p^k-1\mid p^r-1\), which in turn
implies \(k\mid r\). But since \(r\) properly divides \(2k\), the
only possibility is \(r=k\), i.e., \(m/r=2\). Moreover, \(|G|\) is
divisible by \(\frac{p^m-1}{p^r-1}=p^k+1\) and \(|G|\)  divides
\(p^k+1\), so we have in fact \(|G|=p^k+1\), as desired.
\end{proof}

\begin{rem}
\label{fullyramified}
Let \(P\) be a \(p\)-group, and \(N\) a subgroup of \(P\) such that \(P'\leq N\leq\zent P\). For \(\lambda\in\widehat{N}\), set \(Z_\lambda/\ker\lambda=\zent{P/\ker\lambda}\), and observe that \(Z_{\lambda}=\zent\theta\) for every \(\theta\in\irr{P|\lambda}\). In fact, \([\zent\theta,P]\leq N\cap\ker{\theta}=\ker\lambda\) (recall that, \(N\) being central in \(P\), \(\theta_N\) is a multiple of \(\lambda\)), hence \(\zent\theta\leq Z_{\lambda}\); but indeed equality holds, as clearly \([Z_{\lambda},P]\leq\ker\lambda\leq\ker\theta\).

Note also that, if \(\mu\) lies in \(\irr{Z_\lambda|\lambda}\), then \(\mu\) is a character of \(\zent{P/\ker\lambda}\); therefore, for \(\theta\in\irr{P|\mu}\), we have that \(\theta_{Z_\lambda}\) is a multiple of \(\mu\) (and in fact \(\mu\) is fully ramified in \(P\), because \(N\leq\zent\theta\), thus \(P/\zent\theta\) is abelian and \cite[Theorem 2.31]{I} yields \(\theta(1)^2=|P:\zent\theta|=|P:Z_\lambda|\)). Now, taking into account that \(\theta\in\irr{P|\lambda}\) certainly does not vanish on any element of \(\zent\theta=Z_\lambda\), it is easy to see that \(\lambda\in\widehat{N}\) is fully ramified in \(P\) if and only if \(Z_\lambda=N\).
\end{rem}

The proof of next Lemma uses ideas from the proofs of~\cite[Satz
1]{B} and~\cite[Satz~1]{He}.
\begin{lemma}
  \label{f.r.}
Let $P$ be a $p$-group, and \(N\) an elementary abelian subgroup of \(P\) such that \(\frat P\leq N\leq\zent P\). Write $|P/N| = p^n$ and $|N| = p^m$. Assume  $m > n/2$.
Then there are at least $p^{m - \lfloor n/2 \rfloor}$ characters
$\lambda \in \widehat{N}$ such that $\lambda$ is not fully ramified with respect to
$P/N$.
\end{lemma}

\begin{proof}
We can assume that $n$ is even, as otherwise no character $\lambda
\in\widehat{N}$ could be fully ramified with respect to  $P/N$.
Let $\epsilon \in \C$ be a fixed primitive $p$th-root of unity and
$\K = \mathbb{Z}/p\mathbb{Z}$. Then we can associate to every
$\lambda \in\widehat{N}$ an alternating bilinear form $\langle \ , \
\rangle_{\lambda}$ on $P/N$ (as a $\K$-space) by setting, for $aN, bN
\in P/N$, $\epsilon^{\langle aN, bN \rangle_{\lambda}} =
\lambda([a,b])$. From the above remark, it is clear that $\lambda$ is fully ramified with respect to $P/N$ if and
only if the form $\langle \ , \ \rangle_{\lambda}$ is
non-degenerate.

Choosing a basis $\lambda_1, \lambda_2, \ldots, \lambda_m$ of the
dual group $\widehat{N}$ of $N$, there is a bijection between $\K^m$ and
$\widehat{N}$, by associating $\lambda =
\lambda_1^{x_1}\lambda_2^{x_2}\cdots\lambda_m^{x_m} \in \widehat{N}$ to
$({x_1}, {x_2}, \ldots, {x_m}) \in \K^m$; moreover,
$$\langle \ , \ \rangle_{\lambda} = \sum_{i=1}^m x_i\langle \ , \ \rangle_{\lambda_i} \;  .$$
If $A_i$ are the matrices associated to the forms $\langle \ , \
\rangle_{\lambda_i}$ (with respect to a suitable basis of $P/N$),
$\langle \ , \ \rangle_{\lambda}$ is degenerate if and only if
$$d({x_1}, {x_2}, \ldots, {x_m}) = \det{\sum_{i=1}^m {x_i}A_i} = 0.$$
Now, $d({x_1}, {x_2}, \ldots, {x_m}) = f^2({x_1}, {x_2}, \ldots,
{x_m})$ where $f$ is a homogeneous polynomial of degree $n/2$
(see~\cite[(IV), page 46]{G}.
By~\cite[Satz 3]{W}, $f$ has at least $p^{m -n/2}$ roots in $\K^m$
and the result follows.
\end{proof}

We next proceed through a series of results concerning semi-linear actions.

\begin{lemma}
\label{semilinear0} Let \(p\) be a prime, \(V\) a vector space of order \(p^n\), and $H$ a subgroup of  $\Gamma(V)$. Also, setting $X_0 = H \cap \Gamma_0(V)$, let
$\delta$ be a set of primes in $\pi(H) \setminus \pi(X_0)$, let $D$ be a Hall $\delta$-subgroup of $H$. Then $|D|$ divides $n$ and, defining \(k=\dfrac{p^n -1}{p^{n/|D|}-1}\),  the
following facts are equivalent.
\begin{enumeratei}
\item For every $v \in V$, $\cent Hv$ contains a suitable conjugate of $D$.
\item $|\{ D^h : h \in H \}| = k$.
\item $k$ divides $|X_0|$ and $|D|$ is coprime to $p^n -1$.
\end{enumeratei}
\end{lemma}
\begin{proof}
As $D\cap\Gamma_0(V)= (D\cap H)\cap\Gamma_0(V)=D\cap X_0=1$, then $|D|$ divides $n$. Moreover, by
Lemma~3(ii) in \cite{D} we get $|\cent VD| = p^{n/|D|}$. Since $X_0$ acts fixed-point freely  on $V$,  for $v\in V\setminus \{0\}$ we have that \(\cent H v\simeq \cent H v
X_0/X_0\) is cyclic, and therefore $\cent Hv$ contains at most one Hall $\delta$-subgroup of $H$; now the equivalence between (a) and  (b) follows by
counting in $V \setminus \{0  \}$.
Assume now (b). Observing that \(X_0D\) is normal in \(H\) (because \(H/X_0\) is abelian), by the Frattini argument we have \(X_0\norm H D=H\), and (c) follows at once. Conversely, assume (c). If \(r\) is a prime divisor of \(p^{n/|D|}-1\), then $p^{n/|D|} \equiv 1 \pmod r$ and so $k
\equiv |D| \pmod r$; since \(|D|\) is coprime with \(|X_0|\) (hence with \(k\)), it follows that \(|D|\) is coprime with \(p^{n/|D|}-1\), and therefore with \(|\Gamma_0(V)| = p^n -1\). As a consequence,
 $[\Gamma_0(V), D]\simeq\Gamma_0(V)/\cent{\Gamma_0(V)}D$ has order $k$, so $[\Gamma_0(V), D]$ is contained in $X_0$. Finally, \([\Gamma_0(V),D]=[\Gamma_0(V),D,D]\leq[X_0,D]\), hence $k=|[X_0,D]|$, and therefore \(k=|X_0/\cent {X_0} D|=|H:\norm H D|=
|\{ D^h : h \in H \}|,\) as desired.
\end{proof}

\begin{lemma}
\label{semilinear1} Let $H$ act  on a group $A$  
 and let $s$ be a prime divisor of $|H/\cent HA|$. Assume that,
for every $a \in A\setminus\{1\}$,  $\cent Ha$ contains a
Sylow $s$-subgroup of $H$ as a normal subgroup. Then $A$ is an elementary abelian
$p$-group for some prime $p$, and either $p = s = 3$ (and \(|A|=3^2\)), or $H/\cent HA \leq
\Gamma(A)$, $(H/\cent HA) \cap \Gamma_0(A)$ acts irreducibly on $A$, and
\(s\nmid |(H/\cent HA) \cap \Gamma_0(A)|\).
\end{lemma}
\begin{proof}
Set $\overline{H} = H/\cent HA$.
By~\cite[Lemma 4.4]{L}, then either $|A| = 9$ and $s= 3$, or
there exists a normal abelian subgroup $\overline{X_0}$ of $\overline{H}$ such that $A$ is
an irreducible $\overline{X_0}$-module. By~\cite[Theorem 2.1]{MW}, this implies that
$\overline{H} \leq \Gamma(A)$ and that $\overline{X_0} \leq \overline{H} \cap \Gamma_0(A)$. The last
claim is obvious, as $s\mid |\o H|$ and every Sylow \(s\)-subgroup of \(\overline{H}\) centralizes
some non-trivial element of \(A\), whereas every non-trivial element
of \(\overline{H}\cap\Gamma_0(A)\) acts fixed-point freely on \(A\).
\end{proof}

The following fact is folklore, but we include a proof for convenience.

\begin{lemma}
  \label{ppd}
Let $H \leq \Gamma(V)$ be a  group of semilinear
maps on $V$. Let $|V| = p^n$, $p$ prime. Let $T_0$ be a  subgroup  of $H$ such
that $|T_0|$ is   a primitive prime divisor
of $p^n -1$.
Then
$\cent H {T_0} =   H \cap \Gamma_0(V)=\fit H$.
\end{lemma}

\begin{proof}
Let $\mathbb{K} = {\rm GF}(p)$. By~\cite[II.3.10]{H}, $V$ is an irreducible
$\mathbb{K}[T_0]$-module. Hence by Schur's Lemma
$\mathbb{L} = {\rm End}_{\mathbb{K}[T_0]}(V)$ is division ring. Since $\mathbb{L}$ is
finite, it is a field by Wedderburn's Theorem.
So, as $\cent H{T_0}$ is a subgroup of  the multiplicative group of $\mathbb{L}$,
$\cent H{T_0}$ is cyclic. But $\cent H{T_0}$ acts irreducibly on $V$ (as it contains
 $T_0$),
and hence by~\cite[II.3.10]{H} we get that
$\cent H {T_0} \leq  H \cap \Gamma_0(V)$.
So, as $H \cap \Gamma_0(V)$ is cyclic, we conclude that $\cent H {T_0} =  H \cap \Gamma_0(V) \leq \fit H$.
Note that  $t > n$ because $t$ is a primitive prime divisor of $p^n -1$, so a Sylow $t$-subgroup of $H$ is contained in $\cent H{T_0}$, 
hence $\fit H$ centralizes $T_0$. 
\end{proof}

\begin{lemma}
  \label{lemmaB}
Let $H$ be a solvable group, p a prime, and  $V_1$,  $V_2$  two $H$-modules over \(\GF p\).
Assume that there exists a prime $s\in \pi(H/\fit H)\setminus \{p\}$ such that,  for \(i\in\{1,2\}\) and  for every  $v \in V_i \setminus\{0\}$,  $\cent H v$ contains a Sylow $s$-subgroup of $H$ as
a normal subgroup.
Then both  $V_1$ and $V_2$ are irreducible $H$-modules,  $|V_1| = |V_2|$, and 
$H/\cent H{V_i} \leq \Gamma(V_i)$ for \(i\in\{1,2\}\).
\end{lemma}

\begin{proof}
For \(i\in\{1,2\}\), set $|V_i| = p^{n_i}$ and $C_i = \cent H{V_i}$; also, let $S$ be a Sylow $s$-subgroup of $H$.
Note that $s$ divides $|H/C_i|$, as otherwise $S$ would be a characteristic subgroup of $C_i \nor H$ yielding $S \leq \fit H$, against our assumption. 
 
So, by Lemma~\ref{semilinear1} we get that
$H/C_i \leq \Gamma(V_i)$, that $V_i$ is an irreducible $H$-module and that $s$ does not 
divide the order of $X_i/C_i = (H/C_i) \cap \Gamma_0(V_i)$.
Moreover, by Lemma~\ref{semilinear0}, $s$ divides $n_i$ and \((p^{n_i}-1)/(p^{n_i/s}-1)\) divides \(|X_i/C_i|\). 
Observe that there exists a primitive prime divisor \(t_i\) for \((p,n_i)\): otherwise either $n_i= 2$ or $p^{n_i} = 2^6$, and in both cases (as $s \neq p$) $s$ divides 
$(p^{n_i}-1)/(p^{n_i/s}-1)$, contradicting the fact that \(s\nmid|X_i/C_i|\). 
Note also that $t_i \neq s$, as $t_i > n_i$. 
Let $T_i/C_i$ be a subgroup of prime order $t_i$ of $X_i/C_i$; so $T_i/C_i \nor H/C_i$.

We claim that $T_1$ is not contained in  $C_1 C_2$. Otherwise, writing $\o H = H/C_1$, we have $\o{T_1} \leq \o{C_2}$ and, 
choosing a non-trivial $v \in V_2$ such that $S \leq \cent Hv$, both $\o S$ and $\o{T_1}$ are normal subgroups
of $\o{\cent Hv}$, so $[\o S, \o{T_1}] = 1$. But this, by Lemma~\ref{ppd}, implies that $s$ divides $|\o{X_1}|$, a contradiction. 

As $T_1C_2/C_2$ centralizes $T_2/C_2$ (also if $t_1 = t_2$), then again Lemma~\ref{ppd} yields that 
$t_1$ divides $|X_2/C_2|$ and hence $t_1$ divides $p^{n_2} -1$. We conclude that $n_2 \geq n_1$. 

Similarly, one shows that $n_1 \geq n_2$, completing the proof. 
\end{proof}

With the following lemmas, we will gather some relevant information
on the character degree graph of solvable groups.

\begin{lemma}
\label{semilinear2} Let $G$ be a solvable group, and $E$ an abelian normal subgroup of $G$. Assume that $E$ has a complement \(H\) in \(G\) and that $\fit G = E \times Z$ with $Z \leq H \cap \zent G$. 
Setting \(X=\fit H\), let $q \in \pi(X/Z)\setminus\pi(E)$, and
let $s \in \pi(H/Z)\setminus\pi(E)$ be such that $q$ and $s$ are not adjacent in $\Delta(G)$.
Let $Q\in \syl q{X}$,  $S \in \syl sH$ and $L = (QS)^H$ (the normal
closure of \(QS\) in $H$). Then the
following conclusions hold.
\begin{enumeratei}
\item We have \([E,Q]=[E,L]\) and \(\cent EQ= \cent EL\); moreover, $E
 =[E,L]\times \cent EL$.
\item Set \(A=[E,Q]\). Then \(A\) is an elementary abelian \(p\)-group, say of order \(p^n\), where \(p\) is a suitable prime. Also, 
$Z = \cent{LZ}A$ and, for every non-trivial $a \in A$, $\cent {LZ}a$ contains a conjugate of $S$ as a normal subgroup. 
Moreover, $LZ/Z \leq \Gamma(A)$ and $L_0/Z = (LZ/Z) \cap \Gamma_0(A)$ acts irreducibly  on $A$. We also have that $d = |SZ/Z|$ divides
\(n\), and \((p^n-1)/(p^{n/d}-1)\) divides \(|L_0/Z|\).
\item There exists a primitive prime divisor \(t\) of \(p^n-1\).
\item \(L_0/Z=\fit{LZ/Z}\). Moreover,  $LZ = L_0S$ and \(p\) does not divide \(LZ/Z\).
\end{enumeratei}
\end{lemma}

\begin{proof}
Set \(A=[E, Q]\) and \(B=\cent EQ\).
As $q$ does not divide $|E|$ and $E$ is abelian, we have $E = A \times B$. 
Consider now the action of $G$ on the dual group $\widehat{E} = \widehat{A} \times \widehat{B}$. 
For $\alpha \in \widehat{A}\setminus\{1\}$, $q$
divides $|G:\cent G{\alpha}| = |H:\cent H\alpha|$. Also, the linear character $\alpha$ extends to
$\cent G{\alpha}$, because $A$ has  a complement (namely $B \cent H{\alpha}$)  in $\cent G{\alpha}$. 
Thus, by Gallagher's Theorem and Clifford Correspondence, this forces $\cent H\alpha \simeq \cent G{\alpha}/E$  to contain an \(H\)-conjugate 
of $S$ as a normal subgroup (and also, $S$ is abelian).
 Let $\alpha \in \widehat{A}\setminus\{1\}$ be such that $S \leq
\cent H\alpha$ and let $\beta \in \widehat{B}$; then $\cent H{\alpha \times \beta}
= \cent H \alpha \cap \cent H \beta$ and $q$ divides $|H:\cent H{\alpha\times
\beta}|$. As $\alpha \times \beta$ extends to its inertia subgroup in \(G\), using as above Clifford 
Theory and that no irreducible character of $G$ has degree divisible by $qs$, we get that
 the unique Sylow $s$-subgroup $S$ of $\cent H \alpha$
must also be contained in $\cent H\beta$. We conclude that $S$ acts
trivially on $\widehat{B}$ and hence that  $S \leq \cent H{B} \nor H$. Thus $L =
QS^H \leq \cent H{B}$, so that \(B=\cent EL\). Moreover, we get \([E,L]=[A\times B,L]=[A,L]\leq A\), hence $A = [E,L]$ and (a) is proved.

Next, observe that $\cent HE = Z$: in fact, if $x \in \cent HE$, then $x$ centralizes $E Z = \fit G$, 
so $x \in \fit G \cap H = Z$. Thus, it follows that $Z = \cent{LZ}A$.
Since all Sylow $s$-subgroups of $H$ are contained in $L$ and $Z \leq \zent G$, we have that $\cent{LZ}{\alpha}$ contains an
\(L\)-conjugate of $S$ as a normal subgroup, for every $\alpha\in \widehat{A}\setminus\{1\}$. 
Hence, as $s$ is coprime to $|\widehat{A}|$,  an application of
Lemma~\ref{semilinear1} together with Lemma~\ref{semilinear0} yields that  
\(\widehat{A}\) (thus \(A\)) is an elementary abelian \(p\)-group of order \(p^n\), where \(p\) is a suitable prime and \(n\) a suitable integer; moreover, setting $\overline{H} = H/Z$, we get $\overline{L} \leq \Gamma(\widehat{A})$ and $\overline{L_0} = \overline{L} \cap \Gamma_0(\widehat{A})$ acts irreducibly  on $\widehat{A}$. 
We also have that \(s\) does not divide \(|\overline{L_0}|\), whereas $d = |\overline{S}|$ divides
\(n\), and \((p^n-1)/(p^{n/d}-1)\) divides \(|\overline{L_0}|\).

As in the proof of Lemma~\ref{lemmaB},  there exists a primitive prime divisor $t$ of  
$p^n -1$.
Otherwise, either $n = 2$ or $p^n = 2^6$. In both cases, as $s$ and $p$ are distinct primes,  $s$ divides $(p^n -1)/(p^{n/d}-1)$, so $s$ divides $|\overline{L_0}|$, a contradiction. This proves (c), 
and Lemma~\ref{ppd} yields \(L_0/Z=\fit{LZ/Z}\).
In order to conclude the proof of (d), it remains to show that $LZ = L_0S$.

 Clearly \(t\) divides
$(p^n -1)/(p^{n/d}-1)$, hence it divides \(|\overline{L_0}|\). 
Denoting by $\o{T_0}$ the subgroup of  \(\overline{L_0}\) with $|\o{T_0}| = t$,  
by Lemma~\ref{ppd} it follows that $\cent{\overline{L}}{\overline{T_0}} = \overline{L_0}$.
Note that
\(t\) is larger than \(n\) and hence, as \(|\overline{L}/\overline{L_0}|\) divides  \(n\), we get
that   a Sylow $t$-subgroup of $\overline{L}$ is contained in $ \overline{L_0}\).
This implies that \(\overline{Q}\leq \overline{L_0}\), since  $\overline{Q} \nor \overline{H}$ centralizes $\overline{T_0}$.
But now both \(\overline{Q}\) and \(\overline{S}\) lie in \(\overline{L_0S}\);
moreover, as \(\overline{L}/\overline{L_0}\) is cyclic,  \(\overline{L_0S}\) is normal in
\(\overline{H}\). So \(\overline{L}=\overline{L_0S}\) and 
 hence $LZ = L_0S$. Recalling that $s \neq p$ and  $|\o{L_0}|$ divides $p^n -1$, we have also
that \(p\) does not divide \(|LZ/Z|\), and the proof of (d) is complete.

In particular,   the actions of \(\overline L\) on \(A\) and on \(\widehat{A}\) are isomorphic and
also (b) is proved. 
\end{proof}

\begin{lemma}
\label{lemma8} Let $G$ be a solvable group, and assume that $\fit G = M\times U$, where $M\nor G$ is a non-abelian
$p$-group, \(U\nor G\) is abelian, and $p$ does
not divide $|G:\fit G|$. If every irreducible character of $U$ has
an extension to its inertia subgroup in $G$, then either $U\leq\zent
G$ or $d_G(p,v)\leq 2$ for every $v\in\V G$.
\end{lemma}

\begin{proof} Our first claim is
that, if \(t\in\V G\setminus\{p\}\) is not adjacent to \(p\), then
every Sylow \(t\)-subgroup of \(G\) centralizes \(U\). In fact, let
\(\theta\) be in \(\irr M\). Setting \(P=M\times \oh p U\), we have
that \(\theta\times 1_{\oh p U}\in\irr P\) extends to
\(I_G(\theta\times 1_{\oh p U})=I_G(\theta)\) because
\(p\nmid|G:P|\), and therefore \(\theta\) extends to
\(I_G(\theta)\). Now, denoting by \(K\) a complement for \(M\) in
\(G\) containing \(U\), the degrees of the characters in
\(\irr{G|\theta}\) are of the kind
\(|G:I_G(\theta)|\cdot\theta(1)\cdot\lambda(1)\), where
\(\lambda\in\irr {I_K(\theta)}\). As a consequence, if \(\theta\) is
chosen to be non-linear, \(I_K(\theta)\) contains a Sylow
\(t\)-subgroup \(T\) of \(G\), and \(T\) is abelian and normal in
\(I_K(\theta)\). Now, \(U\) is a nilpotent normal subgroup of
\(I_K(\theta)\), thus \([U,T]=1\) as wanted.

Now, let us assume \(\cent G U\neq G\) and let \(w\) be a prime
divisor of \(\fit{G/\cent G U}\); observe that, by the previous
paragraph, \(w\) is a vertex of \(\Delta(G)\) which is adjacent to
\(p\). Also, let \(s\in \V G\) be non-adjacent to \(w\). In this
setting, we claim that \(s\) is adjacent to \(p\) in
\(\Delta(G)\). In fact, we can certainly find \(\phi\in\widehat{U}\) such
that \(w\) divides \(|G:\cent G{\phi}|\). Thus \(\cent G{\phi}\) contains a
Sylow \(s\)-subgroup \(S\) of \(G\); by our assumptions, \(\phi\)
extends to \(\cent G{\phi}\), and so \(\cent G{\phi}/U\) has an abelian
normal Sylow \(s\)-subgroup. It follows that \(SU\nor \cent G{\phi}\)
and  \(\syl s{\cent G{\phi}}=\{S^u\;|\;u\in U\}\). Take
now any \(\theta\in\irr M\); we get that \(I_G(\theta\times\phi)\)
(thus \(I_G(\theta)\)) contains a Sylow \(s\)-subgroup \(S^u\) of
\(G\) for some \(u\in U\), but then \(I_G(\theta)\) contains \(S\)
because \(U\leq I_G(\theta)\). We conclude that \(S\) centralizes
every irreducible character of \(M\), thus it centralizes \(M\) by
coprimality. But this forces \(S\not\leq\cent G U\), which yields
that \(s\) is adjacent to \(p\) in view of the previous paragraph.

To sum up, under the assumption \(U\not\leq\zent G\), we proved the
existence of \(w\in\V G\) which is adjacent in $\Delta(G)$ to \(p\) and to every vertex of
\(\Delta(G)\) not adjacent to \(p\). It easily follows
that every vertex of \(\Delta(G)\) can be reached from \(p\) through
a path of length at most \(2\), and the proof is complete.
\end{proof}

\begin{lemma}
\label{hyper} Let \(G\) be a group, \(p\) a prime, and \(P\) a
normal \(p\)-subgroup of \(G\). If \(G/\cent G P\) is a \(p\)-group,
then \(P\) is a hypercentral subgroup of \(G\).
\end{lemma}
\begin{proof} Since \(G/\cent G P\) is a \(p\)-group, the number of elements
 in \(P\) that are fixed under the action of \(G/\cent G P\) is  divisible by \(p\) (unless \(P\) is trivial, in which case there is nothing to prove),
and hence \(\zent G\cap P\neq 1\); in particular, \(\zent G\) is non-trivial.
 Consider now the factor group \(\overline{G}=G/\zent G\) and adopt the bar convention throughout; clearly \(\overline{P}\) is a normal \(p\)-subgroup
 of \(\overline{G}\) and \(\overline{G}/\cent{\overline{G}}{\overline{P}}\)
is a \(p\)-group (because it is isomorphic to a quotient of \(G/\cent G P\)), therefore we can use induction on the order of the group and conclude that
  \(\overline{P}\leq\hyperzent{\overline{G}}\). The claim now follows, as \(\hyperzent{\overline{G}}=\overline{\hyperzent G}\).
\end{proof}

\begin{lemma}
\label{hyper2} Let \(G\) be a solvable group, and \(p\) a prime.
Setting \(P=\oh p G\) and \(N=P'\), assume that \(P\) is
non-abelian, \(P/N\leq\zent{G/N}\), and that  every irreducible
character of \(P\) has an extension to its inertia subgroup in
\(G\). Then \(p\) is a complete vertex in \(\Delta(G)\).
\end{lemma}
\begin{proof}
Note that, if \(H\) is a $p$-complement of \(G\), then \(H\)
centralizes \(P/N\) and hence it centralizes \(P\) by coprimality;
this yields that \(G/\cent G P\) is a \(p\)-group, thus \(P\) is
hypercentral in \(G\) by Lemma~\ref{hyper}.

Working by contradiction, we assume that $p$ is not a complete vertex of
$\Delta(G)$ and we consider a vertex \(s \in \V G\), $s \neq p$,
such that \(ps\) does not divide any
irreducible character degree of \(G\), and let \(\theta\) be in
\(\irr P\). Since \(\theta\) extends to \(I=I_G(\theta)\), the
degrees of the characters in \(\irr{G|\theta}\) are of the kind
\(|G:I|\cdot\theta(1)\cdot\lambda(1)\), where \(\lambda\in\irr
{I/P}\). As a consequence, if \(\theta\) is chosen to be non-linear,
\(I/P\) contains $SP/P$ where $S$ is a suitable Sylow $s$-subgroup
of $G$ (recall that \(|G:I|\) is a \(p\)-power), and \(SP/P\) is
abelian and normal in \(I/P\); here $S$ is in fact abelian, as
$S\simeq SP/P$. Now, as \(P\leq\hyperzent G\), the nilpotency of
\(SP/P\) yields the nilpotency of \(SP\); moreover, \(SP\) is normal
in \(I\), which is in turn subnormal in \(G\) because \(I/\cent G
P\) is a subgroup of the \(p\)-group \(G/\cent G P\). We conclude
that \(SP\) is a nilpotent subnormal subgroup of \(G\), whence it
lies in \(F=\fit G\). To sum up, \(S\) is a Sylow \(s\)-subgroup of
\(G\) which is abelian and normal in \(G\) (as \(S\leq F\)), hence
\(s\) is not a vertex of                                                                                                                                                                                                                                        \(\Delta(G)\), a contradiction; in other words, every prime
in \(\V G\setminus\{p\}\) is adjacent to \(p\) in \(\Delta(G)\), as
wanted.
\end{proof}

\begin{proposition}
\label{prop9'} Let $G$ be a solvable group such that \(\Delta(G)\)
is connected of diameter $3$, and let $p$ be a prime. Setting $P=\oh p G$, assume
that $P$ is non-abelian, $P'$ is a minimal normal subgroup of $G$ and that $\Delta(G/P')$ is a disconnected
graph. Then $P$ is the Sylow $p$-subgroup of $G$ and $P/P'$ is not contained in the
center of $G/P'$.
\end{proposition}

\begin{proof}
Write $N=P'$ and $F = \fit G$. Observe first that $N$ lies in $\frat P$, thus $N\leq\frat
G$ and the ascending Fitting series of $G/N$ is just the image of
the ascending Fitting series of $G$ under the natural homomorphism
onto $G/N$. 

First, we will  show that $P$ is a Sylow $p$-subgroup of $G$.   
Assume, working by contradiction, that  this is not the case. 

Note that, since $\Delta(G)$ has diameter three, the graph $\Delta(G/N)$ (whose vertex set is  \(\V G\) in this situation)  has no isolated vertices, and therefore  $G/N$ is of type (b) or (c) of Theorem \ref{class}.
Hence, the Sylow subgroups of $G/F$ are all cyclic.

Let \(\theta\) be any character in \(\irr P\). Setting
\(R=\oh{p'}{F}\), clearly we have that \(\theta\times 1_R\) is an
extension of \(\theta\) to \(F\), such that \(I_G(\theta\times
1_R)=I_G(\theta)\). Moreover, since all Sylow subgroups of $G/F$ are
cyclic, \(\theta\times 1_R\) (and therefore \(\theta\)) extends to
\(I=I_G(\theta)\). 
If \(P/N\leq\zent{G/N}\) then, by 
Lemma~\ref{hyper2}, \(p\) would be a complete vertex of \(\Delta(G)\), against the assumption $\diam{\Delta(G)} = 3$.
Hence, \(P/N \not\leq\zent{G/N}\). By Lemma~\ref{lewis}, this yields that \(F/N\) is abelian and  hence that \(G/N\) is
necessarily of  type  (b) of Theorem~\ref{class}.

So
$F/N = M/N \times Z/N$,  where $M/N$ is an elementary abelian $p$-group having a complement $H/N$ in $G/N$  and 
$Z/N = \zent{G/N} = \cent{H/N}{M/N}$.
Set $K = \fitt G$ (so $K/N = \fitt{G/N}$).
Also, $M/N$ is an irreducible $K/N$-module, and 
$K/F$ and $G/K$ are cyclic groups of coprime order (in fact,  $\pi(K/F)$ and $\pi(G/K)$
are the connected components of $\Delta(G/N)$).  Note that $p\not\in \pi(K/P)$ (as $P = \oh pG$).

Let  $\pi_0 $ be the set of vertices not adjacent to $p$ in $\Delta(G)$; so, $\pi_0 \neq \emptyset$.
 We remark that $\pi_0 \sbs \pi(K/P)$, as $p \in
\pi(G/K)$ and $\pi(G/K)$ induces a complete subgraph of $\Delta(G)$.

Write $K = PX$, where $X \leq H$ is a  $p$-complement of $K$.
Note  that $X$ is abelian, because $X/(Z \cap X) \cong K/F$ is cyclic and $Z\cap X$ is central in $X$. 
Let $Y$ be the Hall $\pi_0$-subgroup of $X$.  Then $YN/N \nor H/N$ (as $XN/N = \fit{H/N}$). So  $PY \nor  G$ and hence
$\Delta(PY)$ is a subgraph of $\Delta(G)$. We deduce that
$\Delta(PY)$ is disconnected with components $\{p\}$ and $\pi_0$ (observe that every \(r\in\pi_0\) is a vertex of \(\Delta(PY)\); in fact, a Sylow \(r\)-subgroup \(R\) of \(PY\) is also a Sylow $r$-subgroup of $G$ and, if $R$ is abelian and normal in $PY$, then
the same is true in $G$ (as $PY \nor G$), against the fact that \(r\) is in \(\V G\)), hence $PY$ is of type (a)  in~ Theorem \ref{class}.
So, setting $C = \cent PY$, we have that $N \leq C$ and that every non-linear irreducible 
character of $P$ is fully ramified in $P/C$. Note also that $M\not\leq C$, since otherwise (as above) \(YN/N\) centralizes \(F/N\), so \(Y\leq F\) and no prime divisor of $|Y|$ would be a vertex of $\Delta(G)$.
Let $Z_p/N$  be  the  Sylow $p$-subgroup  of $Z/N$. 
Since   $Y$ acts trivially on both $N$ and $Z_p/N$, then  $Z_p\leq C $.  
Moreover, $C \cap M= N$, as $(C \cap M)/N$ is a proper  submodule of the 
$K/N$-irreducible module $M/N$. 
 Now,  since $P/N = M/N \times Z_p/N$, it follows that $C = Z_p$. 

We next observe that \(M\) is non-abelian. In fact, as \(p\) is an isolated vertex of \(\Delta(PY)\), every non-linear irreducible character of \(P\) is centralized by \(Y\); therefore, an application of \cite[Theorem 19.3]{MW} yields \(N=[P,Y]'\), but \([P,Y]=[MC,Y]=[M,Y]\leq M\), whence \(M'=N\). In particular, \(\Delta(MY)\) is also disconnected with components \(\{p\}\) and \(\pi_0\). Now, working with $MY$ instead of $PY$ (as in the previous paragraph), we get that every non-linear irreducible character of $M$  is fully ramified  with respect  to $M/\cent M Y=M/N$.  By Lemma~\ref{f.r.} this implies $|M/N| \geq |N|^2$.

Consider now $U = \cent XN$.
Assume first $U = X$.
Let $\chi \in \irr G$ such that $N \not\leq \ker(\chi)$ and let $\psi$ be an irreducible constituent of $\chi_P$. 
Then $\psi$ is fully ramified in $P/C$ and $I_G(\psi) = I_G(\theta)$, where $\theta$ is the irreducible constituent of $\psi_C$.
As $X$ acts trivially on both $C/N=Z_p/N$ and $N$, then $X$ centralizes $C$ and hence $\psi$ is $K$-invariant. 
So, recalling that $(|P|, |K/P|) = 1$, $\psi$ extends to $K$. Since $K/P \cong X$ is abelian, Gallagher's Theorem implies that
every irreducible character of $K$ lying over $\psi$ has degree coprime to $|K/P|$. 
It follows that $\pi(\chi(1)) \subseteq \pi(G/K)$ (recall that $p \in \pi(G/K)$). We conclude that
$\Delta(G) = \Delta(G/N)$ is disconnected, a contradiction.

 Hence, $U < X$. Choose $q \in \pi(X/U)$ and let $Q\in \syl q X$. 
So, $QN/N \nor H/N$ and then, in particular, $PQ \nor G$.  

We  remark that $p$  and $q$ are adjacent in $\Delta(G)$. If not, they are not adjacent in the subgraph $\Delta (PQ)$  as well, and
hence $\Delta(PQ)$ is of type (a) of Theorem~\ref{class}, giving  $N \leq \cent PQ$, so $Q \leq U$, against the choice of $q$. 
So, since $(p,q) \in \E G$ and  $q$ is not a complete vertex of $\Delta(G)$, there exists a
vertex $s \neq p$ of $\Delta(G)$ such that $(q,s) \not\in \E G$. Note also that  $s \in \pi(H/Z)$: otherwise, as the $p$-complement of $Z$ is abelian, \(G\) would have an abelian normal Sylow \(s\)-subgroup and \(s\) would not be a vertex of \(\Delta(G)\).

Let $S$ be a Sylow $s$-subgroup of $H$ and let $L = (QS)^H$. 
By applying Lemma~\ref{semilinear2} to $G/N$ with \(M/N\) playing the role of \(E\),  we get that $LZ/Z$ acts as a faithful, irreducible semi-linear group on $M/N = [M/N, Q]=[M/N,L]$ (here we are taking into account that \(\cent{M/N}Q\) is normal in \(G/N\), hence it is trivial because \(M/N\) is an irreducible \(K/N\)-module).
So, recalling that $C/N$ is central in $G/N$ (thus it lies in \(\cent{P/N}L\)), we see that  $\cent{P/N}L=C/N\times\cent{M/N}L = C/N$.
Also, for  every  non-trivial element $a \in M/N$, $\cent{LZ/Z}a$ contains a Sylow $s$-subgroup of $LZ/Z$ as a normal subgroup.
Finally, note that by part (d) of Lemma~\ref{semilinear2}, $\fit{LZ/Z}$ acts fixed-point freely on $M/N$ and hence,
in particular, $s$ does not divide $|\fit{LZ/Z}|$.

Next, observe that \([N,Q]=N\),  as $N$ is minimal normal in $G$ and  \(1 < [N,Q]=[N,PQ] \nor G\).
Hence, setting $B = \cent PL$, we have that $B \cap N =\cent N L\leq\cent N Q=1$. Since $C/N = \cent {P/N}L=\cent{P/N}{LZ/Z}$ and, recalling \ref{prop9'}(d), the action of \(LZ/Z\) on \(P/N\) is a coprime action, we get  $C = NB$ and hence 
$C = N \times B$, because $N$ is central in $P$. 
 
Let now $\gamma \in \widehat{C}$ such that $N \not\leq \ker (\gamma)$. So $\gamma$ is 
fully ramified in $P/C$; let $\psi \in \irr P$ the unique constituent of $\gamma^P$. 
Then $I_G(\psi) = \cent G{\gamma}$.
As $\gamma = \alpha \times \beta$ with $\alpha \in \widehat{N}\setminus\{1\}$ and 
$\beta \in \widehat{B}$, then $q$ divides $|G:\cent G{\gamma}|$. Since $\psi$ extends to 
$\cent G{\gamma}$ (in fact, the 
Sylow $p$-subgroups of $G/P$ are cyclic because $p \not\in \pi(K/P)$) and 
$(q,s) \not\in \E G$, as usual we get that $\cent G{\gamma}/P$ contains a Sylow 
$s$-subgroup of $G/P$ as a normal subgroup; in particular, the same is true also 
for $\cent{LZ/Z}{\gamma}$. 
But $\cent {LZ/Z}{\gamma} = \cent{LZ/Z}{\alpha}$, as $LZ$ acts trivially on $B$. 
So, for every $\alpha \in \widehat{N}\setminus\{1\}$, $\cent {LZ/Z}{\alpha}$ contains a Sylow $s$-subgroup of $LZ/Z$ 
as a normal subgroup. 
Hence an application of Lemma~\ref{lemmaB} yields that $|M/N| = |\widehat{N}|  = |N|$, a contradiction.

So far, we have shown that $P$ is a Sylow $p$-subgroup of $G$. Thus, every irreducible character of $P$ has an
extension to its inertia subgroup in $G$ and hence Lemma~\ref{hyper2} yields that $P/N$ is not central in $G/N$.
This finishes the proof.
\end{proof}

\begin{lemma}
  \label{C}
Let $P$ be a non-abelian normal Sylow $p$-subgroup of a solvable group $G$ and let
$H$ be a $p$-complement of $G$. Assume that there is a prime divisor $s$ of $|H/\fit H|$ such that $s$ is not adjacent to $p$ in $\Delta(G)$. Then for all $2 \leq i  \leq c$,
where $c$ is the nilpotency class of $P$, the factor groups
$M_i=\gamma_i(P)/\gamma_{i+1}(P)$ are chief factors of $G$ of the same order \(p^n\), with $n \geq 3$, and \(G/\cent G{M_i}\) embeds in \(\Gamma(p^n)\).
\end{lemma}

\begin{proof}
For \(2 \leq i \leq c\), 
take any non-trivial \(\mu\) in \(\widehat{M_i}\): by
\cite[Theorem 13.28]{I}, and by the fact that \(M_i\) is central in
\(P/\gamma_{i+1}(P)\), there exists \(\phi\) in
\(\irr{P/\gamma_{i+1}(P)\mid\mu}\) such that \(I_H(\phi)=\cent
H{\mu}\). As \(i\geq 2\) and \(\mu\) is non-trivial, clearly
\(\phi(1)\) is a multiple of \(p\). Now, viewing \(\phi\) as a
character of \(P\) by inflation, we have that \(\phi\) extends to
its inertia subgroup in \(G\) (by coprimality); as a consequence of our non-adjacency
assumption, Clifford Correpondence together with Gallagher's Theorem yield
that \(I_H(\phi)\) (thus \(\cent H{\mu}\)) contains a Sylow
\(s\)-subgroup of \(H\) as a normal subgroup. Observe also that
\(\cent H{M_i}\) does not contain any Sylow \(s\)-subgroup of \(H\):
in fact, if \(S\in\syl s H\) lies in \(\cent H{M_i}\) (which in turn
lies in \(\cent H{\mu}\)), then \(S\) would be a characteristic
subgroup of \(\cent H{M_i}\), and therefore a normal subgroup of
\(H\), against the fact that \(s\) is a divisor of \(|H/\fit H|\).
We are then in a position to apply Lemma~\ref{lemmaB}, which
yields that all groups \(\widehat{M_i}\) are in fact  irreducible \(H\)-modules
of the same order;
thus  all  \(M_i\) are irreducible \(H\)-modules as well (i.e., \(M_i\)
is a chief factor of \(G\)), all of them have the same order \(p^n\), and all the groups \(G/\cent G {M_i}\) embed in \(\Gamma(p^n)\).
Moreover, Lemma~\ref{semilinear0} yields that $s$ divides $n$ and that $s$ is coprime to $p^n -1$. 
As $s \neq p$, this implies that $n \geq 3$. 
\end{proof}

\section{The main results}

In what follows, we write $\diam{\Delta(G)} \geq 3$ to indicate that
either $\Delta(G)$ is connected and ${\rm diam}(\Delta(G)) = 3$, or
that $\Delta(G)$ is not connected (i.e., ${\rm diam}(\Delta(G)) =~
\infty$).

\begin{proposition}
\label{maxiemendamento?} Let $G$ be  solvable group  such that   \(\fit G=P\) is   a non-abelian Sylow
\(p\)-subgroup of \(G\), and assume that  \(P'\) is a minimal normal subgroup of
\(G\).
If $\diam {\Delta(G)} \geq 3$ and  \(p\) is not an isolated vertex of $\Delta(G)$,  then $\fitt G$ does not centralize $P'$.
\end{proposition}
\begin{proof}
We denote by $H$ a $p$-complement of  $G$, and we set $X = \fit H$ and $N = P'$. Since  $\fitt G=PX$, we have to prove that  $X$ does not centralize $N$.
 Working by contradiction  we assume  that $X$ centralizes $N$, and go through a series of steps.

\begin{enumeratei}
\item \emph{Every prime in $\pi(H/X)$ is adjacent to $p$ in $\Delta(G)$.}

\begin{proof} Let $q \in \V G$ be such that $(q,p) \not\in \E G$ and
consider a non-trivial character $\phi \in \widehat{N}$. As $\phi$ is
$X$-invariant and $p$ does not divide $|X|$, by  ~\cite[Theorem 13.28]{I}
there is a $\psi \in \irr{P|\phi}$ such that $X \leq I = I_G(\psi)$.
By coprimality, $\psi$ extends to $I$;  hence  Gallagher's  Theorem and Clifford
Correspondence imply that $I \cap H \simeq  I/P$ contains a Sylow
$q$-subgroup $Q$ of $G$,  and  that $Q$ is abelian and normal in $I
\cap H$. As $X \nor I \cap H$, then $Q$ centralizes the
$q$-complement of $X$. But $Q$ centralizes also the Sylow
$q$-subgroup $Q\cap X$ of $X$, as $Q$ is abelian. Thus, $Q$ centralizes
$X = \fit H$ and hence $Q \leq X$, so $q \not\in \pi(H/X)$.
\end{proof}

Note that a group \(G\) as in our hypotheses, in the disconnected case, is as in (c) of Theorem~\ref{class}; in particular, the connected components of
$\Delta(G)$ are $\{p \} \cup \pi(X)$ and $\pi(H/X)$, clearly against what obtained in the paragraph above. This contradiction settles the disconnected case, so we may henceforth assume that \(\Delta(G)\) is connected of diameter \(3\).

\smallskip
\item \emph{There exist $q \in \pi(X)$ and $s \in \pi(H) \setminus \pi(X)$ such that
$q$ is non-adjacent to both $p$ and $s$ in $\Delta(G)$.}

\begin{proof}
By Lemma~\ref{brodkey}  the subgraph of \(\Delta (G)\) induced
by \(\pi(X)\) is complete  and by (a) the set of vertices of $\Delta(G)$ that are not adjacent to $p$ is contained in  \(\pi(X)\).  
We may consider a  path \(q- r- s_1- s_2\) in $\Delta (G)$ connecting two vertices with distance $3$, where $q \neq p$, $q$ is not adjacent to $p$, and  one among
the \(s_i\) is not $p$: naming it $s$ we have the claim. 
\end{proof}

\smallskip
\item \emph{Let $Q \in \syl qX$, $S \in \syl sH$. Set $M = [P,Q]$, $L =
(QS)^H$ and $X_0=  \fit L$. Then $M' = N$ and  $M/N$ is an irreducible
$X_0$-module. Moreover,
setting $|M/N| = p^n$, we have that \(X_0\) is a cyclic group whose
order is divisible by $(p^n -1)/(p^{n/|S|} - 1)$.}

\begin{proof}
Note that $Q$ stabilizes every non-linear irreducible character of
$P$, as $Q \nor H$ and $q$ is not adjacent to $p$. So $M' = P' = N$
by~\cite[Theorem 19.3]{MW}, and it also follows that $M/N = [P/N,
Q]$. Now we are in a position to apply Lemma~\ref{semilinear2} to
the group \(G/N\), with respect to the non-adjacent vertices \(q,s\)
of \(\Delta(G/N)\), and all the desired conclusions follow.
\end{proof}

\smallskip
\item \emph{$V = M/N$ is a self-contragredient $X_0$-module.}

\begin{proof}
Let $\psi$ be a non-linear irreducible character of \(M\), and
$\lambda$ an irreducible constituent of \(\psi_N\). Then $\lambda$
is $X_0$-invariant (as \(X_0\leq X\)), and clearly it does not
extend to $M$. Since $M/N$ is irreducible as an $X_0$-module, then
$\lambda$ is fully ramified with respect to  $M/N$ (see \cite[Exercise~6.12]{I}).
So by Remark~\ref{fullyramified} the bilinear form  defined by
$\epsilon^{\langle aN,bN \rangle} = \lambda([a,b])$ on $M/N$ (where
$\epsilon$ is a given primitive $p$-th root of unity) is
non-degenerate, and it is also $X_0$-invariant, as $X_0$ acts
trivially on $N$. Hence, it induces an isomorphism of $X_0$-modules
between $M/N$ and its contragredient module.
\end{proof}
\end{enumeratei}

Note that $q \in \pi(X_0) \sbs \pi(X)$, so Lemma~\ref{brodkey}
implies that $s$ does not divide $|X_0|$. We conclude by applying
Lemma~\ref{modules}, which gives  $p = 2 = s$, a contradiction.
\end{proof}

The next result, which is the core of this work, shows that actually
no group \(G\) such that \(\Delta(G)\) is connected can satisfy the
assumptions of Proposition~\ref{maxiemendamento?}.

\begin{proposition}
\label{maxiemendamento}Let $G$ be a solvable group  such that   \(\fit G=P\) is   a non-abelian Sylow
\(p\)-subgroup of \(G\), and let \(N=P'\) be a minimal normal subgroup of
\(G\). Assume also that $\diam {\Delta(G)} \geq 3$ and \(p\) is not an isolated vertex of $\Delta(G)$.
Then $\Delta(G)$ is not connected, $[P,G]/N$ is a chief factor of $G$, and $|[P,G]/N| = |N|$; moreover, if $H$ is a $p$-complement of $G$, then 
$H \leq \Gamma([P,G]/N)$ and $H/\cent HN \leq \Gamma(N)$.
\end{proposition}

\begin{proof}
Let $H$ be a $p$-complement of $G$, and set $X = \fit H$ (observe that $\cent H P = 1 = \cent H{P/N}$). By the previous proposition, there exists a prime divisor $q$ of $|X|$ such that $Q =
\oh qX$ does not centralize $N$. As \(N\) is minimal normal in
\(G\), we have \([N,Q]=N\) and \(\cent N Q=1\), whence \(Q\not\leq \cent H{\lambda}\) for every non-trivial
\(\lambda\in\widehat{N}\). So \(Q\) does
not lie in the inertia subgroup of any non-linear irreducible
character of \(P\) (as restrictions to $N \leq \zent P$ are homogeneous), and hence $q$ is adjacent to $p$ in $\Delta(G)$.

Since $\V{G/N} = \V{G} \setminus \{p\}$ and $q$ is adjacent to $p$ in $\Delta(G)$, there certainly
exists $s\in \V{G/N}$ that is not adjacent to $q$ in $\Delta(G/N)$.
Observe also that $s$ does not divide $|X|=|\fit {HN/N}|$, as otherwise,    by Lemma~\ref{brodkey}, $q$ would be adjacent to $s$ in $\Delta (G/N)$; therefore  a Sylow \(s\)-subgroup \(S\) of \(H\) is not normal in \(H\). Let $L =
(QS)^H$, \(V=[P/N,Q]\) and $C/N = \cent {P/N}Q$; then, as $N \le [P,Q] =R$, $V = R/N$, $P/N = R/N \times C/N$. Setting \(|V|=p^n\), by applying Lemma~\ref{semilinear2} we obtain that  $C/N = \cent {P/N}L$, \(R = [P,L]\),  $L \leq \Gamma(V)$ and  $V$ is an irreducible
$X_0$-module, where $X_0 = L \cap \Gamma_0(V) = \fit L$; moreover, \(|S|\) divides
\(n\), the order of \(X_0\) is divisible by \((p^n-1)/(p^{n/|S|}-1)\), and there exists a primitive prime divisor \(t_0\) of \(p^n-1\). As in the previous proposition, we shall proceed through a number of steps.

\smallskip
\begin{enumeratei}
\item\label{noiso} \emph{$R$ is not abelian and, as an $X_0$-module, $N$ has no irreducible constituent isomorphic to  $V= R/N$.}
\begin{proof} Suppose, by contradiction, that $R$ is
abelian, and consider the action of \(L\) on the dual group
\(\widehat{R}\): by coprimality, no non-trivial element of
\(\widehat{R}\) is centralized by \(Q\), and therefore every
irreducible character of \(RL\) whose kernel does not contain \(R\)
has a degree divisible by \(q\). Since \(RL\nor G\) (therefore
\(\Delta(RL)\) is a subgraph of \(\Delta(G)\)) and every irreducible
character of \(R\) extends to its inertia subgroup in \(RL\), we
have that \(\cent L{\lambda}\) contains a unique Sylow \(s\)-subgroup of
\(L\) for every \(\lambda\in\widehat{R}\setminus\{1_R\}\). An
application of Lemma~\ref{semilinear1} yields that $\widehat{R}$
(thus \(R\)) is an irreducible $L$-module, contradicting the fact
that $N$ is a proper non-trivial $L$-invariant subgroup of $R$. Therefore $R$ is not abelian.

Observe next that $\langle\; ,\; \rangle: R/N \times R/N \rightarrow N$,
defined by
 $\langle aN, bN \rangle = [a,b]$, for $a,b\in R$,   is a ${\rm GF}(p)$-bilinear map. This induces a homomorphism (of
${\rm GF}(p)$-spaces) $\delta: R/N \otimes_{\GF p} R/N
\rightarrow N$, which is  easily checked to be  an
$X_0$-homomorphism. 
Since $R$ is not abelian and \(N\) is minimal normal in \(G\), we have $R'= N$, whence 
 $\delta$ is surjective. Thus, $\delta$ induces a surjective
$X_0$-homomorphism from $R/N \wedge_{\GF p} R/N$ to $N$, because the symmetric tensors are
in $\ker{\delta}$. If $N$, as an $X_0$-module,  has
an irreducible constituent isomorphic to $R/N$ then, by Lemma~\ref{modules}, $p = 2= s$, which is not our case.
\end{proof}
\item \emph{$\cent PQ$ is an abelian direct factor of \(G\), and \(H\) acts faithfully on $V$.}

\begin{proof}
Note that \(\cent PL=\cent P Q\) and, by coprimality, \(C=\cent P L\times N\); moreover, we have
\(R\cap\cent P L=1\) because \(N=[N,Q]\). As
\(P=R\cent P L\), we get that \(\cent P L\simeq P/R\) is abelian.

If $C$ is not contained in $\zent P$ then, as \(N\leq\zent P\), we
can choose $y \in \cent PL\setminus\zent P$. The map $\varphi_y: R/N
\rightarrow N$, defined by $\varphi_y(aN) = [a,y]$, is then 
a homomorphism of $X_0$-modules, because $X_0$
centralizes $y$. Since \(R/N\) is an irreducible \(X_0\)-module,
\(\varphi_y\) is either injective or the trivial homomorphism. But
in the latter case \(y\) would centralize both \(R\) and $C$, whence 
\(y\in\zent P\). We conclude that $\varphi_y$ is injective, which means that 
the $X_0$-module $N$ has a constituent isomorphic to $R/N$, against step (a).

Hence $C \leq \zent P$. Let $M = \cent PL$; so $P = R \times M$ and
$HR$ is a complement for $M$ in $G$. We next show that $M \leq \zent G$.

Write $\o G = G/N$.
If $M$ is not central in $G$, then
 $\o K = \cent {\o{HR}}{\o M} < \o{HR}$ (note that $ \o L \leq \o K$) and so there exists a prime divisor $t$ of $|\fit{\o{HR}/\o K}|$. Then, clearly, $t$ divides
$|\o G:I_{\o G}(\lambda)|$ for some non-trivial
irreducible character $\lambda$ of $\o M$, whence $t$ divides $\chi(1)$ for all $\chi \in
\irr{\o G|\lambda}$. But $\o{MK} = \o M \times \o K \nor \o G$,
so $t$ is adjacent
in $\Delta(\o G)$ to every vertex of $\Delta(\o K)$.
As  $s$
is a  vertex of $\Delta(\o K)$ (note that $\o S$ is not normal in $\o H$, so
it is not normal in $\o L$, and the same holds in $\o K$), we conclude
that $s$ is adjacent to $t$.
This is true for every vertex $s$ that is not adjacent to $q$ in $\Delta(G/N
)$
and for $q$ as well; in fact,  $\o Q \leq \o K$ and $\o K$ does not have a normal Sylow $q$-subgroup, as otherwise  $\o Q$ would centralize $\o R=V$. It thus follows  that $\Delta(G/N)$ is connected; in particular, \(\Delta(G)\) is connected as well (and it has diameter \(3\), by our assumptions), otherwise \(p\) would be an isolated vertex of \(\Delta(G)\).

Now, by Theorem~~\ref{Zuccari}, $\diam{\Delta(G/N)} \leq 2$, because $\fit{G/N} = P/N$ is abelian. We then have that  there exists $s' \in \V G$
such that $d_G(s', p) = 3$.
Note that then $s'$ is not adjacent to $q$ in $\Delta(G)$ (as $p$ is adjacent
 to
$q$); therefore, $s'$ is not adjacent to $q$ in $\Delta(\o G)$  as well,
and so, as observed above,  it is adjacent to $t$ in $\Delta(\o G)$ and hence  in $\Delta(G)$.
But $M \times K \nor G$ implies  that $p$ is adjacent to $t$ in
$\Delta(G)$,  yielding  $d_G(s, p) \leq 2$. This contradiction shows that \(M\) lies in \(\zent G\).
Hence, $G = HR \times M$ and $M = \cent P Q$ is a central direct factor of $G$.

As a consequence, we get \(\cent H V=\cent H{P/N} = 1\), so $H$ acts faithfully on $V$ and the proof of (b) is complete.
\end{proof}
\noindent In the following, we set \(Y=\cent H N\), write $\o H = H/Y$, and adopt the bar
convention.
The next step settles, in particular, the first claim of the statement.

\medskip
 \item\label{N} \emph{ $[P,G]/N$ and $N$ are chief factors of $G$ of the same order $p^n$, $H \leq \Gamma([P,G]/N)$ and $\o H \leq \Gamma(N)$.}

\begin{proof} By the previous step:
 $$[P,G]=[[P,Q]\cent P Q,G]=[[P,Q],G]\leq[P,Q]=R.$$ 
Therefore, \([P,G] = [P,Q]= R\). It follows that if \(\mu\) lies either in \(\widehat{V}\setminus\{1\}\) or in \(\widehat{N}\setminus\{1\}\), then \(\cent H\mu\) does not contain \(Q\) and, by \cite[Theorem 13.28]{I}, there exists \(\theta\in\irr{P|\mu}\) such that \(\cent H\mu\leq I_H(\theta)\). But the restriction of \(\theta\) to \(R\) (respectively, to \(N\)) is a multiple of \(\mu\), so in fact \(\cent H\mu=I_H(\theta)\); now, since \(\theta\) extends to \(I_G(\theta)\) (and recalling that \(s\) is a vertex of \(\Delta(G)\) not adjacent to \(q\)), we get that \(\cent H\mu\) contains a Sylow \(s\)-subgroup of \(H\) as a normal subgroup. Hence, by Lemma~\ref{lemmaB} we conclude that $V= R/N$ is a chief factor of $G$ of order $|N|$, and that $H \leq \Gamma(V)$ and $H/\cent HN \leq \Gamma(N)$.
\end{proof}

\medskip
In view of the above paragraph, our aim for the rest of the proof will be to show that \(\Delta(G)\) cannot be connected under our hypotheses. To this end, we  assume that $G$ is a counterexample of minimal order; thus $\Delta(G)$ is connected, ${\rm diam}(\Delta(G)) = 3$ and, by minimality, \(G\) has no non-trivial abelian direct factors. In particular, step (b) yields $\cent PQ=1$, $R = P$ and  $V = P/N$.

Setting $m = |H/X|$, as $H \leq \Gamma(V)$ we  have that \(m\) divides
\(n\); in particular, a primitive prime divisor $t_0$ of $p^n -1$ (which exists, as observed before), being   larger than \(n\), does not divide \(m\). Therefore, denoting by $T_0$ a
Sylow $t_0$-subgroup of $H$, we have that $T_0$ lies in $X_0$ and it is in fact central in $X$.  Now, Lemma~  \ref{ppd} yields $\cent H{T_0}\leq \Gamma_0(V)$, whence $X\leq \Gamma_0(V)$ and $X$ acts fixed-point freely on $V$; also,  as observed in the paragraph preceding (a), $s$ does not divide  the order of $X=\cent H{T_0}$. Now (with the notation introcuced before point (c)),

\medskip
\item\emph{\(t_0\) does not divide \( |Y|\), \(Y\leq X\), and $\o X= \o H \cap \Gamma_0(N)$.} 

\begin{proof} As observed above,  \(|T_0|\mid |\o H\cap\Gamma_0(N)|\). This in turn implies \(t_0\nmid|Y|\); thus (as \(Y\) and \(T_0\) are both normal in \(H\)) we get \([Y,T_0]\leq Y\cap T_0=1\), whence \(Y\leq\cent H{T_0}=X\).
Thus,  set $U/Y = \o H \cap \Gamma_0(N)$. We clearly have that \(\o X\) centralizes \(\o T_0\), so Lemma~\ref{ppd} yields \(\o X\leq \o H \cap \Gamma_0(N)\). On the other hand we get
$[U,T_0] \leq Y \cap T_0$ because, again by Lemma~\ref{ppd}, \(\o T_0\) is contained in the cyclic group \(\o U\); thus $U \leq \cent H{T_0} = X$, and so $\o X= \o H \cap \Gamma_0(N)$.\end{proof}

\smallskip
\item  \emph{The subgraphs of $\Delta(G)$ induced on $\pi_1 = \{ p\} \cup
\pi(X)$ and $\pi_2 = \pi(m)$ are  complete graphs. Hence, in
particular, $\pi(X) \cap \pi(m) = \emptyset$ and $X$ has a complement $D$ in $H$.}

\begin{proof}
As $H \leq \Gamma(V)$ and $X=\fit  H$, then $H/X$ is nilpotent (in fact, cyclic); thus,
Lemma~\ref{brodkey} implies that both $\pi(X)$ and $\pi(H/X)= \pi(m)$
induce complete subgraphs of $\Delta(G)$.

It remains to show that $(p,t) \in \E G$ for all $t \in \pi(X)$. Let
$T \in \syl tX$; as $[V, T] \leq V$ is $X$-invariant and
non-trivial, we see that $[V, T] = V$. If $(p,t) \not\in \E G$
then, since \(PT\nor G\), the  graph $\Delta(PT)$ is disconnected, and \(PT\) is as in case (a) of Theorem~\ref{class}. In particular,  $N \leq \cent PT$ and every non-linear irreducible character of \(P\) is fully ramified with respect to \(P/\cent P T\). On the other hand, by coprimality, $\cent PT/N = \cent {V}T$ is trivial,  and in fact every non-linear irreducible character of $P$ is fully ramified with respect to  $P/N.$ In this setting, an application of Lemma~\ref{f.r.} yields the contradiction
$|P/N| \geq |N|^2$.
\end{proof}

\smallskip
\item \emph{The graph \(\Delta(G/N)\) is disconnected with connected components
$\pi(X)$ and $\pi(D) = \pi (m)$, where $D$ is a complement for $X$ in $H$.
Moreover, $(p^n -1)/(p^{n/|D|} -1)$ divides $|X|$.}
\begin{proof}
Arguing by contradiction, assume that $\Delta(G/N)$ is connected. Then, by  Theorem~\ref{Zuccari}, $\Delta(G/N)$ has diameter (at most) two, and therefore a pair of vertices at distance $3$ in $\Delta (G)$ must include the prime $p$. Let $(p,v)$ be such a pair and 
$p- t- r- v$ a shortest path connecting them; then step (e) yields $r,v \in \pi(D)=\pi (m)$ and $t \in \pi(X)$. Take now any \(\lambda\in\widehat{N}\setminus\{1\}\); then $p$ divides $
\theta (1)$ whenever \(\theta\) lies in \(\irr{P|\lambda}\), and if \(r\) divides \(|\o H:\cent{\o H}\lambda|\), then \(r\) divides \(|\o H:\cent{\o H}\theta|\), and so \(pr\) divides \(\chi(1)\) for every \(\chi\in\irr{G|\theta}\), a contradiction. Therefore, Lemma~\ref{semilinear0} yields that $k_r = (p^n
-1)/(p^{n/|R|} -1)$ divides \(|\o X|\), where \(R\) is a Sylow \(r\)-subgroup of \(H\).

On the other hand, consider $\chi \in \irr G$ such that $tr$ divides $\chi(1)$. Then $p$
does not divide $\chi(1)$, whence $\chi \in \irr{G/N}$. Let $\mu\in\irr{P/N}$ be an irreducible constituent of \(\chi_P\); note that $\mu \neq 1_{P/N}$, as $t\mid\chi(1)$.
Since $r\mid\chi(1)$, then $r$ divides $|H:I_H(\mu)|$, and again
Lemma~\ref{semilinear0} implies that $k_r$ does not divide $|X|$, against what observed in the previous paragraph.

Thus, $\Delta(G/N)$ is disconnected and, by Lemma~\ref{brodkey}, it is clear that
its connected components are $\pi(X)$  and $\pi(D)$.
Consider now $\mu \in \irr{P/N} \setminus \{ 1_{P/N}\}$; as $X$ acts fixed-point freely on
$P/N$, we have that $I_H(\mu) \cap X = 1$, hence, for every $\chi \in \irr{G|\mu}$,
the degree of $\chi$ is divisible by all the primes in $\pi(X)$.
As a consequence, $I_H(\mu)$  must contain a conjugate of $D$ and the last claim
follows by Lemma~\ref{semilinear0}.
\end{proof}

\smallskip
\item \emph{Let $\lambda \in\widehat{N}\setminus\{1\}$.
Then there exists $\theta_0 \in \irr{P|\lambda}$ such that $I_H(\theta_0) = \cent H{\lambda}$.
For every $\theta \in \irr{P|\lambda}\setminus\{\theta_0\}$, we have  $I_H(\theta) \cap X = 1$, and \(I_H(\theta)\) contains a complement \(D\) for \(X\) in \(H\).
}

\begin{proof}
The first claim follows from~\cite[Theorem 13.28]{I}, and from the fact that $I_H(\theta) \leq \cent  H{\lambda}$ for every $\theta \in \irr{P |\lambda}$,
as $N \leq \zent P$.

Consider now a character $\theta \in \irr{P|\lambda}\setminus\{\theta_0\}$.
Recalling that $X/Y$ acts fixed-point freely on $\widehat{N}$, in order to prove
that
$I_H(\theta) \cap X = 1$ it is enough to
show that $Y_0 =I_H(\theta) \cap Y = 1$.
If $Y_0 >1$, then \(\cent V{Y_0}\) cannot be the whole \(V\); since \(V\) is irreducible \(X\)-module and $\cent{V}{Y_0}$ is \(X\)-invariant,  \(\cent V{Y_0}\) is trivial and, by~\cite[Exercise 13.10]{I}, $\theta_0$ is the only character in $\irr{P|\lambda}$ which
is $Y_0$-invariant, a clear contradiction. We conclude that $Y_0 = 1$.

Now, if \(I_H(\theta)\) does not contain any complement for \(X\) in \(H\) (i.e., any Hall \(\pi(m)\)-subgroup of \(H\)), then there exists a prime \(r\in\pi(m)\) which does not divide \(|H:I_H(\theta)|\); as a consequence, any \(\chi\in\irr{G|\theta}\) would be such that \(pr|X|\) divides \(\chi(1)\). This yields a contradiction, as \(r\) would be a complete vertex of \(\Delta(G)\), and also the last claim is proved.
\end{proof}

For the next two steps of the proof, it will be convenient to introduce some specific notation. We define $\nfr N$ as the set of all $\lambda \in \widehat{N}\setminus\{1\}$ that are not fully ramified in \(P\); since $|P/N| = |N|$, Lemma~\ref{f.r.} ensures that \(\nfr N\) is not empty. We shall also take into account Remark~\ref{fullyramified} and the notation introduced therein; in particular recall that, for \(\lambda\in\widehat{N}\), the subgroup \(Z_\lambda\) is defined by \(Z_\lambda/\ker\lambda=\zent{P/\ker\lambda}\).

\smallskip
\item \label{Z}
\emph{Let $\lambda$ be in $\nfr N$, and set $M = Z_{\lambda}$.
Then the following conclusions hold.
\begin{enumeraten}
\item There exists a complement $D$ of $X$ in $H$ such that
$\cent H{\lambda} = YD$.
\item \(YD\) normalizes \(M\) and, for every $a \in M\setminus N$, we have
$\cent {YD}{aN} = D^y$ for some $y \in Y$.
\item If \(D_1\) is a complement for \(X\) in \(H\) such that \(\cent{M/N}{D_1}\) is non-trivial, then \(D_1\leq YD\).
\item If $Y \not\leq \zent H$, then
$YD/\cent{YD}{M/N} \leq \Gamma(M/N)$ and $M/N$ is an irreducible $Y$-module.
\end{enumeraten}}

\begin{proof}
Since \(\lambda\) is not fully ramified in \(P\), then \(|\irr{P|\lambda}|>1\), and by (f) we  find \(\theta\in\irr{P|\lambda}\) such that \(I_H(\theta)\) contains a complement \(D\) for \(X\) in \(H\). Now, \(I_H(\theta)\) is contained in \(\cent H \lambda\) because \(\theta_N\) is a multiple of \(\lambda\), and therefore \(D\leq\cent H{\lambda}\). Clearly \(Y\) lies in \(\cent H{\lambda}\) as well, and since $X/Y$ acts fixed point freely on $\widehat{N}$ (and \(\lambda\neq 1\)), we deduce that \(\cent H{\lambda}\) is in fact \(YD\).

As for (ii), observe that \(YD=\cent H\lambda\) normalizes \(\ker\lambda\), so it acts on \(P/\ker\lambda\) and on \(\zent{P/\ker\lambda}=M/\ker\lambda\) as well; as a consequence, \(\cent H\lambda\) normalizes \(M\). Since \(\lambda\in\irr{N/\ker\lambda}\) and \(N/\ker\lambda\) is a subgroup of the abelian group \(M/\ker\lambda\), we have that \(\lambda\) extends to \(M\), and in fact \(\irr{M|\lambda}\) consists of extensions of \(\lambda\). By \cite[ Theorem 13.28]{I}, among those extensions we can choose \(\mu_0\) that is \(\cent H\lambda\)-invariant, and we can write \(\irr{M|\lambda}=\{\mu_0\rho\;|\;\rho\in\widehat{M/N}\}\). Consider now \(\rho\in\widehat{M/N}\), and take \(\theta\in\irr{P|\mu_0\rho}\). Since \(\theta\) lies in \(\irr{P|\lambda}\), step (f) together with the previous paragraph yield that \(I_H(\theta)\) contains \(D^x\) for some \(x\in X\); but \(I_H(\theta)\leq \cent H\lambda=YD\), therefore there exists \(y\in Y\) such that \(D^x=D^y\). Moreover,  recalling that \(\theta_M\) is a multiple of \(\mu_0\rho\), we get \(D^y\leq I_{YD}(\theta)\leq I_{YD}(\mu_0\rho)\) and, as \(\mu_0\) is \(D^y\)-invariant, we easily deduce that \(D^y\) lies in \(\cent {YD}\rho\) as well. If \(\rho\neq 1\), taking into account that \(X\) acts fixed-point freely on \(P/N\), we also have \(\cent {YD}\rho\cap X=1\). Thus \(\cent {YD}\rho=D^y\). Now, by coprimality, (ii) follows.

Assume now that \(D_1\) is a complement for \(X\) in \(H\) such that \(aN\in \cent{M/N}{D_1}\), where \(a\) lies in \(M\setminus N\). By (ii), \(aN\) is also centralized by \(D^y\) for a suitable \(y\in Y\); but, as \(X\cap\cent H {aN}=1\), \(\cent H {aN}\simeq\cent H{aN}X/X\) is cyclic, and therefore \(D_1=D^y\leq YD\).

Finally, if \(Y\not\leq\zent H\), then \([Y,D]\neq 1\) and therefore \(D\) does not lie in \(\cent{YD}{M/N}\) (otherwise, as \(Y\) acts fixed-point freely on \(M/N\), we would have \(D=\cent{YD}{M/N}\nor YD\)), so (iv) follows by Lemma~\ref{semilinear1}.
\end{proof}

\item\label{YinCenter}  $Y \leq \zent H$.

\begin{proof}
Aiming at a contradiction, let us assume $[Y,H] \neq  1$. Thus, as $Y\le X$ and $X$ is abelian,  $[Y,D] \neq 1$ for any complement $D$ for $X$ in $H$.

Observe first that every $x \in P$ is contained in a subgroup $Z_{\lambda}$ for
some $\lambda \in \nfr N$. In fact, as this clearly holds for \(x\in N\), let us focus on an element \(x\not\in N\). Since \(\cent P x\supseteq N\langle x\rangle\supset N\), we get \(|[P,x]|=|P|/|\cent P x|<|P|/|N|=|N|\), and it is enough to choose a non-trivial \(\lambda\in\irr{N/[P,x]}\) in order to have \(x\in Z_\lambda\) (with \(\lambda\in\nfr N\), as \(x\in Z_\lambda\) and so \(Z_\lambda>N\)).
Observe also that, if $Z_{\lambda_1} \neq Z_{\lambda_2}$ for \(\lambda_1,\lambda_2\in\nfr N\), then $Z_{\lambda_1} \cap Z_{\lambda_2} = N$, since
both $Z_{\lambda_1}/N$ and $Z_{\lambda_2}/N$  are irreducible $Y$-modules by (g).

Given $\lambda \in \nfr N$, we set \(M=Z_\lambda\) and we denote by $\Xi(M)$ the dual group
of $N/[P, M]$ (as a subgroup of the dual group $\widehat N$ of $N$). Note that if \(\mu\neq 1_N\) lies in \(\Xi(M)\), then one has \(M\leq Z_\mu\), thus \(\mu\) lies  in \(\nfr N\) and so in fact, by irreducibility of the \(Y\)-module \(Z_\mu/N\),  \(M=Z_\mu\). In other words, we have \(\Xi(M)\setminus\{1_N\}=\{ \mu \in \nfr N: Z_{\mu} = M\}\).
It is then clear that 
$\Xi(M) \cap \Xi(Z_{\nu}) = \{ 1_N \}$,  if $Z_{\nu} \neq M$ for some $\nu \in \nfr N$. 

Also, given a complement $D_1$ for $X$ in $H$,  if $\cent{M/N}{D_1} >1$
then, for every $\lambda \in \Xi(M)$, by (g) $D_1 \leq \cent H{\lambda}$ and hence
$\Xi(M) \leq \cent{\widehat N}{D_1}$.

So, we set $\cent H{\Xi(M)} = YD$, for a suitable complement $D$ of $X$  in $H$.  
Now set $|M/N| = p^t$ (note that, by Clifford's Theorem, \(t\) does not depend on \(\lambda\in\nfr N\), as $M/N$ is an irreducible $Y$-module and $Y \nor H$), and let $L/N$ be a complement for the irreducible \(Y\)-module $M/N$ in
$P/N$.
Observe that, for $a, b \in P$ and $y \in Y$, we have
$$[a^y, b] = [a^y, b]^{y^{-1}} = [a, b^{y^{-1}}].$$
Therefore, for any $b \in M \setminus N$,
\begin{equation}
[L, M] = \langle [a,b^y]: a\in L,\; y \in Y\rangle =
\langle [a^{y^{-1}},b] : a\in L,\; y \in Y\rangle = [L, b].
\end{equation}
As a consequence, we get
$|[L, M]| = |L|/|\cent Lb| \leq |L/N| = p^{n-t}$;
taking into account that $|N| = p^n$, this yields
$|N/[L, M]| \geq p^t = |M/N|$.

Now, by Lemma~\ref{f.r.}, there are at least
$p^{t/2}$ characters in $\widehat{N/[L, M]}$ that are not fully ramified
in $M/[L, M]$. We claim that all the non-trivial characters in $\widehat{N/[L, M]}$ that are not fully ramified
in $M/[L, M]$, are in $\Xi(M)$.
In fact, given a $\nu \in \widehat{N/[L, M]}$ such that $\nu$ is not
fully ramified in $M/[L, M]$, then (recalling Remark~\ref{fullyramified}) there exists an element $b \in M \setminus N$
such that $[M, b] \leq \ker{\nu}$. As $[L, M] \leq \ker{\nu}$, then
$[P, b] = [LM, b] \leq \ker{\nu}$.
As in (1) (with \(P\) in place of \(L\)) one sees that $[P, b] = [P, M]$, and the claim is proved.
In particular, we get $|\Xi(M)| \geq p^{t/2} - 1$.

Finally, set $d = |YD:\cent{YD}{M/N}|$. Since, by point (g),  $YD/\cent{YD}{M/N}$ is an irreducible subgroup of $\Gamma(M/N)$, we have $|\cent{M/N}D| = p^{t/d}$, where $d = |D:C_D(M/N)|$.
Noting  that, by what observed before,
$$\mathcal{Z} = \{(Z_{\lambda}/N)\setminus\{1\} : \lambda \in \nfr N\;\; {\rm{ and }}\;\;
\cent{Z_{\lambda}/N}D > 1 \}$$ is a partition of $\cent{P/N}D\setminus\{1\}$, we conclude that
$$ |\mathcal{Z}| = \dfrac{|\cent{P/N}D|-1}{p^{t/d} -1} =\dfrac{p^{n/m} -1}{p^{t/d} -1}.$$
Since, as observed, $\Xi(Z_{\lambda}) \leq \cent{\widehat{N}}D$ for
every $Z_{\lambda}$ such that $(Z_{\lambda}/N)\setminus\{1\} \in \mathcal{Z}$, and $|\Xi(M)| \le p^{t/2}-1$ , we deduce $$p^{n/m}-1 =  |\cent{\widehat{N}}D\setminus\{1\}| \geq
\frac{p^{n/m} -1}{p^{t/d} -1} \cdot (p^{t/2} -1).$$
Hence $p^{t/2}-1 \leq p^{t/d} -1$, which implies $d = 2$ (so \(|D|\) is even). In particular \(p\neq 2\), thus \(p^n-1\) is an even number as well as \(|D|\). 
But, by  Lemma~\ref{semilinear0} the numbers \(p^n-1\) and \(|D|\) must be coprime. This contradiction completes step~(h).
\end{proof}

\medskip
\item  \emph{Final  contradiction.}

\begin{proof}
By step (f),  the graph $\Delta(G/N)$ is disconnected  with connected components $\pi(X)$ and
$\pi(D)$, where $D$ is a complement for $X$ in $H$.

Since step (h) yields $Y \leq \cent XD$, we have that $|Y|$ divides $p^{n/|D|} -1$.
Now,  $p^{n/|D|} -1$ is coprime to $k = (p^n -1)/(p^{n/|D|} -1)$ and, as \(k\) divides \(|X|\) by (f), we see that $k$ divides the order of $\o X = \o H \cap \Gamma_0(\widehat{N})$.
Thus Lemma~\ref{semilinear0} yields that, for every $\lambda \in \widehat{N}$, a conjugate
of $D$ lies in $\cent H{\lambda}$. In particular, for every $\lambda \in \widehat{N} \setminus \{ 1 \}$, we get
$\cent H{\lambda} = Y \times D^x$ for a suitable $x \in X$.

Now, let $\chi \in \irr{G}$ be such that \(\chi_N\) has a non-trivial irreducible constituent \(\lambda\), and let $\theta \in \irr{P|\lambda}$ be an irreducible constituent of $\chi_P$.
Thus, $I_H(\theta) \leq \cent H{\lambda} = Y \times D^x$ (for some $x \in X$) is cyclic.
So, as $\chi(1) = \psi(1)|G:I_G(\theta)|$ for a suitable
$\psi \in \irr{I_G(\theta)|\theta}$, we conclude that
$\chi(1) =  \theta(1)|H:I_H(\theta)|$ and hence, taking into account that \(I_H(\theta)\) contains a conjugate of \(D\) by step (f), the prime divisors of $\chi(1)$ lie in $\{p\} \cup \pi(X)$.
Therefore $\Delta(G)$ is disconnected, a contradiction.
\end{proof}
 \end{enumeratei}

The proof is now complete. \end{proof}

The following result, together with Remark~\ref{spiegazione} and Lemma~\ref{C} for what concerns the dimension $n$ of the factors $M_i$,  will yield Theorem~A and, as a by-product,  Theorem~C. For this reason we do not include an independent proof for Theorem~C, that could be obtained with a direct and much shorter argument.

\begin{theorem}
\label{main}
Let $G$ be a solvable group such that either \(\Delta(G)\) is connected of diameter \(3\), or \(\Delta(G)\) is disconnected. In the disconnected case, assume also that \(F=\fit G\) is non-abelian and that, whenever \(\oh r G\) is  non-abelian, the prime \(r\) is not an isolated vertex of \(\Delta(G)\). Then the following conclusions hold.
\begin{enumeratei}
\item Let $p$ be a prime. If $\oh p G$ is non-abelian, then it is a Sylow $p$-subgroup of $G$.
\item There exists a unique prime $p$ such that $P=\oh p G$ is non-abelian.
 Also, denoting by $U$ the  $p$-complement of $F$, we have $U\leq\zent G$.
\item $\Delta (G/\gamma_3(P))$ is disconnected, and \(G/F\) is a non-nilpotent group whose Sylow subgroups are all cyclic. If $c$ is the nilpotency class of $P$, all factors $M_1 = [P,G]/P'$ and $M_i = \gamma_i(P)/\gamma_{i+1}(P)$, for $
2\le i\le c$, are chief factors of $G$ of the same order $p^n$. Moreover, for all $1\le i\le c$,  $G/\cent G{M_{i}}$ embeds in $\Gamma (p^n)$ as an irreducible subgroup.
\end{enumeratei}
\end{theorem}
\begin{proof} Let $G$ be as in the assumptions and $F = \fit G$. Observe that if $\Delta (G)$ is connected then, by Theorem~\ref{Zuccari}, there exists a prime \(p\)
such that \(P=\oh p G\) is non-abelian (thus $F$ is not abelian in any case). 
We argue by induction, and thus assume that $G$ is a counterexample of minimal order. 

If $M$ is a normal subgroup of $G$ such that $M \leq \frat G$, then the Fitting series of $G/M$ is the image of the Fitting series of $G$ under the natural homomorphism onto $G/M$. 
Moreover, if $\fit{G/M} = F/M$ is non-abelian and \(\V G=\V{G/M}\), then all assumptions on $G$ are inherited by $G/M$. 

\smallskip
\noindent Observe first that, if \(\Delta(G)\) is disconnected, then \(G\) is as in (c) of Theorem~\ref{class}; in particular, parts (a) and (b) are already known to be true.  

Thus, as it concerns the proof of (a) and (b), we may assume that \(\Delta(G)\) is connected. Let $p$ be a prime such that \(P=\oh p G\) is non-abelian. Setting \(N=P'\), observe
that $N\leq\frat G$.  Then, in order to prove (a) and (b), we may assume that $N$ is a minimal normal subgroup of $G$. In fact, if $M$  is a normal subgroup
of $G$ such that $1 < M < N$, then $\oh p {G/M} = P/M$ is non--abelian and \(\V G=\V{G/M}\), whence, as observed before, (a) and (b) hold in $G/M$.  In particular, $P$ is a Sylow $p$-subgroup of $G$, and so (a) holds in $G$. 
Also, if $U$ is the $p$-complement of $F$, then $ UM/M$ is the $p$-complement of $\fit{G/M}$; property (b) in $G/M$ and normality of $U$ in $G$ yield $[U, G] \leq M \cap U = 1$, and (b) holds true in $G$. 

Assuming thus $N$ to be a minimal normal subgroup of $G$, we start proving claims (a) and (b) for $G$ under the additional
hypothesis that \(\Delta(G/N)\) is disconnected. 
Then, again denoting by \(U\)
the $p$-complement of \(F\), Proposition~\ref{prop9'} yields  \(P/N\not\leq\zent{G/N}\), and Lemma~\ref{lewis} ensures that \(F/N\) is abelian and \(UN/N\) is central in \(G/N\). Thus, we get \([G,U]\leq N\cap
U=1\), and (b) is proved in $G$. Since Proposition~\ref{prop9'} also
ensures that \(p\nmid|G/P|\), then (a) is achieved as well.

For the proof of (a) and (b), we may therefore assume that  \(\Delta(G/\oh
p G')\) is connected for every prime \(p\) such that \(\oh p G\) is
non-abelian.
Thus,  let \(p\) be  such a prime and, again, write \(P=\oh p G\) and
 \(N=P'\). 
 
Suppose, by contradiction, that (a) does not hold in $G$, that is, \(p\) divides \(|G/P|\).
Then \(\V{G/N}=\V{G}\) and, since \(\Delta(G/N)\)
is connected, diam(\(\Delta(G/N)\))=\(3\). In
particular, \(F/N\) is non-abelian by  Theorem~\ref{Zuccari}, and therefore there
exists a prime \(q\ne p\) such that \(\oh q{G/N}\) is non-abelian.  Now, $G/N$ satisfies the hypotheses of the Theorem and so, by choice of $G$, $G/N$ satisfies (a), (b) and (c): in particular, \(P/N\) is
central in \(G/N\) and all the Sylow subgroups of \(G/F\) are cyclic. 
Setting \(R=\oh{p'}{F}\),  and taking any  
\(\theta\) in \(\irr P\), we have, as in the second paragraph of the proof of Proposition~\ref{prop9'}, that
\(\theta\times 1_R\) is an extension of \(\theta\) to \(F\), such
that \(I_G(\theta\times 1_R)=I_G(\theta)\). Since the Sylow
subgroups of $G/F$ are cyclic, \(\theta\times 1_R\) (and therefore
\(\theta\)) extends to \(I=I_G(\theta)\). We may then apply
Lemma~\ref{hyper2} and get the contradiction that  \(p\) is a complete vertex of
\(\Delta(G)\). Hence, $G$ satisfies (a).

We move next to (b). First, we prove the following claim: \emph{if
\(p\) and \(q\) are two different primes such that both \(P=\oh p
G\) and \(Q=\oh q G\) are non-abelian, then the diameters of
\(\Delta(G/P')\) and \(\Delta(G/Q')\) are both at most \(2\).} (Note
that the hypothesis of this claim forces \(p\) and \(q\) to be
adjacent in \(\Delta(G)\).) In fact, assume that one of those graphs (which are connected by what proved before),
\(\Delta(G/P')\) say, has diameter \(3\). Then, arguing as in the last paragraph, 
we have that  \(P/P'\) is central in \(G/P'\). Let \(H\) be a
 $p$-complement of \(G\), then $G = PH$ by (a) and \(H\) centralizes
\(P/P'\), hence $H$ centralizes \(P\) by coprimality. Therefore,
\(G=P\times H\), so \(p\) is a complete vertex of \(\Delta(G)\),
which a contradiction. 
We thus conclude that the diameters of \(\Delta(G/P')\) and \(\Delta(G/Q')\) are
both at most \(2\), as claimed. 
But in this situation, the only
vertices of \(\Delta(G)\) that may have distance \(3\) between each
other turn out to be \(p\) and \(q\), which on the other hand are clearly adjacent.

Thus, we have proved that there exists a unique prime \(p\) such
that \(\oh p G\) is non-abelian. It remains to show that
 the $p$-complement \(U\) of \(F\) is central in $G$.

We claim  that every irreducible
character of \(U\) extends to its inertia subgroup in \(G\). Since $U$ is an abelian normal subgroup of $G$, this is certainly the case if $U$ admits a complement. Otherwise, \(N_0=U\cap\frat G \neq 1\). In this case,
\(\fit{G/N_0}\) is clearly non-abelian and \(\V{G/N_0}=\V{G}\). Thus \(G/N_0\) inherits our assumptions and,  by
choice of $G$, the Sylow subgroups
of \(G/F\) are cyclic, a fact ensuring that also in this case
every irreducible
character of \(U\) extends to its inertia subgroup in~\(G\). 

We are therefore in a position to apply Lemma~\ref{lemma8} and get the desired conclusion unless
 $p$ has distance at most $2$ from every other vertex of
$\Delta(G)$. But this would force (as we know that is \(\Delta(G/N)\)
connected)  $\Delta(G/N)$ to have diameter
$3$, which is against  Theorem~\ref{Zuccari}, because the Fitting
subgroup of $G/N$ is abelian. The proof that (b) holds in $G$  is complete.

\smallskip
Finally,  we prove that (c) holds in $G$. Let $U$ be the $p$-complement of $F$
(thus $U \leq \zent G$ by part~(b)); since $\fit{G/U} = F /
U=PU/U$ is non-abelian, the graph $\Delta(G/U)$ has the same set of vertices as
$\Delta(G)$, and therefore \(G/U\) inherits the assumptions. Furthermore, the projection $G
\rightarrow G/U$ induces a $G$-isomorphism  $P \rightarrow PU/U$,
and $\gamma_i(PU/U) = \gamma_i(P)U/U$.

We claim that, by choice of $G$,  \(U\) is trivial. In fact, suppose \(U\neq 1\); then, by induction on the order of the
group, the conclusions concerning the factors \(M_i\) and the actions of the groups \(G/\cent G{M_i}\) on them are easily achieved, and we
also get that $\Delta(G/\gamma_3(P)U)$ is disconnected. Now, since
\(U\leq\zent G\) and \(\gamma_3(P)\leq\frat G\), we have
\(\fit{G/\gamma_3(P)U}=F/\gamma_3(P)U\). As the Sylow
\(p\)-subgroup of this nilpotent factor group is non-abelian, we get
\(\V{G/\gamma_3(P)U}=\V G\), whence
\(G/\gamma_3(P)U\) is a group as in (c) of Theorem~\ref{class}. In particular, \(G/F\) is a non-nilpotent group whose Sylow subgroups are all cyclic, and
every irreducible character of \(\fit{G/\gamma_3(P)}=F/\gamma_3(P)\) extends to its inertia subgroup in
\(G/\gamma_3(P)\). We are then in a position to apply Lemma~\ref{U} to
obtain the identity  $\Delta(G/\gamma_3(P))=\Delta(G/\gamma_3(P)U)$, and so
\(\Delta(G/\gamma_3(P))\) is disconnected, as wanted. Thus, $U=1$.

We now observe that we may further reduce to the case \(\gamma_3(P)=1\). In fact, suppose \(\gamma_3(P)\ne 1\); then,  setting \(\overline{G}=G/\gamma_3(P)\) and adopting the bar convention,  we have  \(\V{\overline G}=\V G\) and the group \(\overline{G}\) satisfies our hypotheses. Hence, the choice of $G$ 
yields that \(\Delta(\overline{G})\) is disconnected, and the following conclusions follow: \(G/F\) is a non-nilpotent group whose Sylow subgroups are all cyclic, \(M_1\) and \(M_2\) are \(G\)-chief factors of the same order \(p^n\), and \(G/\cent G{M_i}\) embeds in \(\Gamma(p^n)\) as an irreducible subgroup for \(i\in\{1,2\}\). Finally, as \(G/\gamma_3(P)\) is now as in (c) of Theorem~\ref{class}, the prime \(p\) turns out to be adjacent in \(\Delta(G)\) to every prime divisor of \(|\fit H|\) (it cannot therefore be adjacent
to all vertices in \(|H/\fit H|\), as otherwise it would be a
complete vertex in \(\Delta(G)\)). An application of Lemma~\ref{C}
yields now what is left of claim (c) for $G$.

Then $\gamma_3(P)=1$. Suppose, by contradiction, that there exists   \(M\nor G\) with
\(1<M<P'\). Then the graph \(\Delta(G/M)\) is disconnected: in fact, the group \(G/M\) satisfies our hypotheses and therefore, setting \(\overline{G}=G/M\) and adopting again the bar
convention, we have that
\(\Delta(\overline{G}/\gamma_3(\overline{P}))\) is disconnected; but
\(\gamma_3(\overline{P})\) is trivial, so
\(\Delta(\overline{G}/\gamma_3(\overline{P}))=\Delta(G/M)\).
Now the factor group \(G/M\) must be as in Theorem~\ref{class}(c), hence $G$ fulfills the assumptions of Lemma~\ref{C}, an application of which  yields that  $P' = \gamma_2(P)/\gamma_3(P)$ is minimal normal in $G/\gamma_3(P)=G$, against the assumption on  \(M\).

In conclusion, $P'$ is a minimal normal subgroup of $G$ and \(G\) satisfies the hypotheses of Proposition~\ref{maxiemendamento}, so all the desired conclusions follow. Hence,  the proof that $G$ also satisfies (c) is complete, 
but this is a contradiction by the choice of $G$. This completes the proof of the theorem.
\end{proof}

\begin{rem}
\label{spiegazione} Let \(G\) be a solvable group such that
\(\Delta(G)\) has diameter \(3\). Then, assuming the notation and
the conclusions of Theorem A, \(G/\gamma_3(P)\) is a group as in part (c) of  Theorem~\ref{class}, and so \(\Delta(G/\gamma_3(P))\) (which has the
same vertices of \(\Delta(G)\)) is disconnected with components
$\pi_1 = \{ p \} \cup \pi (\fitt G/\fit G)$ and $\pi_2 = \pi (G/\fitt G)$. Thus,
both \(\pi_1\) and \(\pi_2\) induce complete subgraphs of
\(\Delta(G)\). Note that $|\pi_2| \geq 2$, as otherwise $\Delta(G)$ would have a complete vertex. 
Also, \(\Delta(G/P')\) is a disconnected graph  
with components subgraphs 
\(\pi_1\setminus\{p\}\) and \(\pi_2\). 

Now, we have that  $G = PH$, with $P$ a $p$-group, $\gamma_3(P)\ne 1$, and $H/\cent HP$ embeds in $\Gamma (p^n)$.
So, setting \(d=|G/\fitt G|\), an application of Lemma~\ref{semilinear0} yields that \(d\) divides \(n\), \((p^n-1)/(p^{n/d}-1)\) divides \(|\fitt G/\fit G|\), $(p^n)^3$ divides $|P|$ and \(d\) is coprime to \(p^n-1\). As a consequence, on one hand
\cite[Theorem 5.1]{L} yields \(|\pi_1 \setminus \{p\}|\geq2^{|\pi_2|}-1\). On the other hand, we easily get $2 \not \in \pi_2$; in particular, as \(|\pi_2|\geq 2\), we also get that $n$  is divisible by two odd primes, as stated in part (d) of Theorem~A.   

Also, we note that \(d_G(p,v)\leq 2\) for every \(v\in\V G\). In
fact, assume that \(d_G(p,v) = 3\) and let  \(p-t-r-v\) be a  path  in \(\Delta(G)\);
then we get \(t\in\pi_1\), \(r\in\pi_2\) and, given \(\chi\in\irr
G\) such that \(tr\mid\chi(1)\), the prime \(p\) does not divide
\(\chi(1)\). But now \(\chi\) is in fact in \(\irr{G/P'}\) and hence
\(t\) is adjacent to \(r\) in \(\Delta(G/P')\), a contradiction.

Finally, we sketch the proof  that Lewis' example in \cite{Lew2} is of the smallest possible order. 
In fact, as observed above, $|G|$ is a multiple of 
$$(p^n)^3\cdot\frac{p^n-1}{p^{n/d}-1}\cdot d\; . $$
The smallest value of such an integer is attained for $p=2$ and $n=d= 15$, that is $2^{45}\cdot (2^{15}-1)\cdot 15$,  which is precisely the order of Lewis' group.  \end{rem}
We conclude with a proof of Corollary~B.
\begin{proof}[Proof of Corollary B] In the setting of Theorem A, we have that \(G/F\) is a group whose Sylow subgroups are all cyclic. This implies that \(G/F\) is a metacyclic group and, since it is not nilpotent, the desired conclusion follows.
\end{proof}


%
%
%

\newpage

\end{document}